\documentclass[11pt]{amsart}

\usepackage{listings}
\usepackage{geometry}   
\usepackage{setspace}
\usepackage[pdftex]{graphicx}
\usepackage{amssymb, amsthm}
\usepackage{epstopdf}
\makeatletter
\newtheorem*{rep@theorem}{\rep@title}
\newcommand{\newreptheorem}[2]{%
\newenvironment{rep#1}[1]{%
 \def\rep@title{#2 \ref{##1}}%
 \begin{rep@theorem}}%
 {\end{rep@theorem}}}
\makeatother

\newtheorem{theorem}{Theorem}[section]
\newreptheorem{theorem}{Theorem}
\newtheorem{corollary}[theorem]{Corollary}
\newtheorem{lemma}[theorem]{Lemma}
\newreptheorem{lemma}{Lemma}
\newtheorem{proposition}[theorem]{Proposition}  
\newreptheorem{proposition}{Proposition}

\theoremstyle{definition}
\newtheorem{definition}[theorem]{Definition}

\newtheorem{examples}[theorem]{Examples}

\newcommand{\Z}{\ensuremath{{\mathbb{Z}}}}

\newcommand{\R}{\ensuremath{{\mathbb{R}}}}
\newcommand{\E}{\ensuremath{{\mathbb{E}}}}

\newcommand{\bbP}{\ensuremath{{\mathbb{P}}}}
\newcommand{\G}{\Gamma}

\newcommand{\AG}{A_\G}            

\newcommand{\m}{\text{Min}}

\title{Compactifying the Space of Length Functions of a Right-angled Artin Group}
\author{Anna Vijayan}

\begin{document}

\maketitle

\begin{abstract}
\noindent
Culler and Morgan proved that the length function of a minimal action of a group on a tree completely determines the action. As a consequence the space of minimal actions of $F_n$ on trees, up to scaling (also known as Outer Space), embeds in  $\bbP^\infty$ via the map sending an action to its projectivized length function.  They also proved that the image of this map has compact closure.

For a right-angled Artin group $\AG$ whose defining graph is connected and triangle-free, we investigate the space of minimal actions of $\AG$ on 2-dimensional CAT(0) rectangle complexes. Charney and Margolis showed that such actions are completely determined by their length functions; hence this space embeds into $\bbP^\infty$. Here it is shown that the set of projectivized length functions associated to these actions has compact closure.  
\end{abstract}

\section{Introduction}

By definition, right-angled Artin groups (RAAGs) are groups that admit a presentation with a finite number of generators and with all relations commutator relations between pairs of generators. RAAGs are often specified using a finite simplicial graph $\G$, called the defining graph, which determines the presentation. 
Let $\G$ be a finite simplicial graph (recall that a simplicial graph has no loops or bigons). Then the RAAG $A_\G$ is the group whose generators are the vertices of $\G$ and whose relations are the commutator relations between pairs of adjacent vertices. 

Droms \cite{Dr} showed that any two graphs that define isomorphic RAAGs must be isomorphic as graphs; thus there is a one-to-one correspondence between isomorphism classes of RAAGs and isomorphism classes of finite simplicial graphs.

\begin{examples} 
If $\G$ is a discrete graph on $n$ vertices, that is, $\G$ has no edges, then none of the generators commute, so $\AG$ is a free group $F_n$. 

If $\G$ is a complete graph on $n$ vertices, then all generators commute, so $\AG$ is a free abelian group $\Z^n$. 

If $\G$ is a join of $k$ subgraphs $\G_1, \dots, \G_k$ with vertex sets $V_1, \dots, V_k$, then all vertices of $V_i$ commute with all vertices of $V_j$ (for all $i, j$) , so $\AG \cong A_{\G_1} \times \dots \times A_{\G_k}$. In particular, if $\G$ is the graph shown in Figure~\ref{fig:1}, 
\begin{figure}[htbp]
  \centering
  \includegraphics[width=3in]{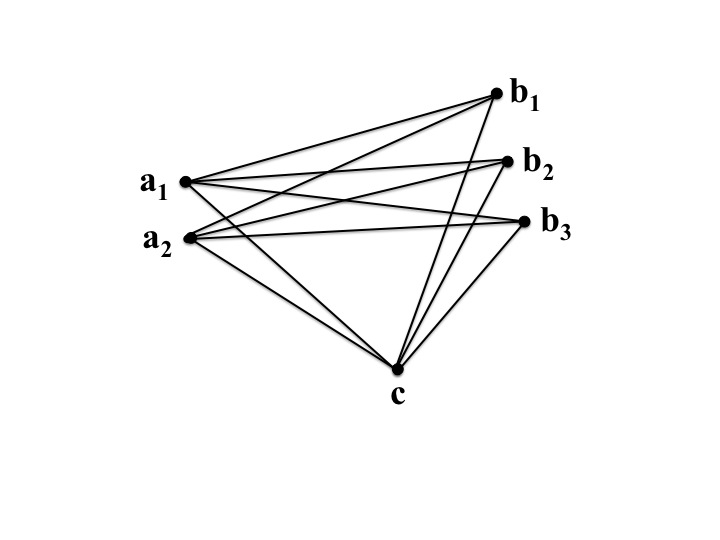}
  \caption{A defining graph $\G$.}
  \label{fig:1}
\end{figure}
then  $\AG=\langle a_1, a_2 \rangle \times \langle b_1, b_2, b_3 \rangle \times \langle c \rangle =F_2 \times F_3 \times \Z$, and 
if $\G$ is a complete bipartite graph $K_{m,n}$, then 
$\AG$ is the direct product of free groups $F_m \times F_n$. 

If $\G$ is a cyclic graph of length $n\geq 5$, then $\AG$ is not a direct product or free product of subgroups. More generally, this is the case for any $\G$ that is connected and does not decompose as a join.
\end{examples}

We have seen that free groups and free abelian groups are extreme examples of RAAGs- one corresponding to a graph with no edges and the other to a graph that includes all possible edges. Clearly many RAAGs lie between the two, since every isomorphism class of graphs defines a different RAAG. We think of RAAGs as interpolating between free groups and free abelian groups.

RAAGs appear as subgroups of other well-known groups. For example, every RAAG embeds as a subgroup of some mapping class group. Right-angled Coxeter groups have presentations like those of RAAGs, except that they include relations giving each generator order two. Davis and Januszkeiwicz \cite{DJ} showed that every RAAG embeds as a finite-index subgroup of some right-angled Coxeter group. Since Coxeter groups are linear, so are RAAGs (linearity of RAAGs was also shown by \cite{HW}.) Thus RAAGs are residually finite, since they are linear and finitely generated. RAAGs are also known to be residually torsion-free nilpotent \cite{DK}; as a consequence they are bi-orderable \cite{DK} and torsion-free. 

RAAGs are known to contain several different classes of groups as subgroups. For example, they contain all of the surface groups (either orientable or not, either closed or with boundary), except for three (\cite{CrW}, extending the work of \cite{DrSS}). They also contain the graph braid groups \cite{CrW}. More recently,  Wise \cite{W} showed that any word-hyperbolic group with a quasiconvex hierarchy has a finite-index subgroup that embeds as a subgroup of a RAAG. 

This last result was instrumental in the proof of Thurston's Virtually Fibered Conjecture for hyperbolic 3-manifolds, which says that every closed hyperbolic 3-manifold has a finite cover that is a surface bundle over the circle. Agol \cite{A} had shown that every right-angled Coxeter group contains a finite index subgroup satisfying the technical property called residually finite rationally solvable (RFRS) \cite{A}; in other words, right-angled Coxeter groups are virtually RFRS. Since the property of being virtually RFRS is inherited by subgroups, and since RAAGs are subgroups of right-angled Coxeter groups,  any subgroup of a RAAG is virtually RFRS; thus Wise's result showed that any word-hyperbolic group with a quasiconvex hierarchy is virtually RFRS. Agol \cite{A} also had shown that any irreducible 3-manifold with RFRS fundamental group is virtually fibered. Wise used this result to show that all closed hyperbolic Haken 3-manifolds are virtually fibered \cite{W, W2}. Finally Agol, Groves, and Manning completed the proof that the Virtually Fibered Conjecture holds for \emph{all} closed hyperbolic 3-manifolds by showing that the virtually Haken conjecture holds for closed hyperbolic 3-manifolds \cite{AGM}.

In Wise's work, a central role is played by Salvetti complexes. Given a RAAG $\AG$, its Salvetti complex is constructed as follows. The 1-skeleton is a rose: it has one vertex and $n$ edges, each of length one, labeled by the generators $v_1, \dots, v_n$ of $\AG$.  For each commutator relation $v_iv_j=v_jv_i$, a unit square 2-cell is attached along the $v_i$ and $v_j$ edges, forming a torus. For every three generators that pairwise commute, a 3-torus is formed by attaching a unit cube along the three 2-tori associated to the three generators. This process is continued, so that an $n$-torus is formed for every $n$ generators that pairwise commute. The resulting finite-dimensional cubed complex $S_\G$ is the Salvetti complex for $\AG$. Its universal cover  $\tilde{S}_\G$  is a CAT(0) cubical complex, and the RAAG $\AG$ acts geometrically on  $\tilde{S}_\G$. My thesis concerns more generally certain kinds of geometric actions of RAAGs on CAT(0) rectangle complexes.

We pause to introduce some terminology for certain subgraphs associated with the vertices of a defining graph. Let $v$ be a vertex of a defining graph. Then the \emph{link}, $lk(v)$, of $v$ is the full subgraph spanned by the vertices adjacent to $v$. The \emph{star}, $st(v)$, of $v$ is the full subgraph spanned by $v$ and the vertices adjacent to $v$.

It is sometimes possible to prove that some property holds for a restricted set of RAAGs, where the restriction is put on the defining graph. A useful idea has been that of a \emph{homogeneous} graph. Homogeneity is defined using induction. We say that a finite simplicial graph is \emph{homogeneous of dimension 1} if it is non-empty and discrete (contains no edges). Furthermore, we say that a finite simplicial graph is \emph{homogeneous of dimension $n>1$} if it is connected and if the link of every vertex is homogeneous of dimension $n-1$. Note that a graph $\Gamma$ is homogeneous of dimension 2 if and only if it is connected and triangle-free; this condition implies that the Salvetti complex $S_\Gamma$ is 2-dimensional.

The outer automorphism groups of RAAGs are objects of considerable interest. If the automorphism group of a RAAG $\AG$ is denoted by $Aut(\AG)$, and the subgroup formed by the conjugations by $Inn(\AG)$, then the outer automorphism group is the group $Out(\AG):=Aut(\AG) / Inn(\AG)$. Just as RAAGs interpolate between free groups and free abelian groups, their outer automorphism groups interpolate between the outer automorphism group of a free group, $Out(F_n)$, and $GL_n(\Z)$. Several properties shared by both  $Out(F_n)$ and $GL_n(\Z)$ have also been shown to hold for the outer automorphism group of an arbitrary RAAG. For example, Charney and Vogtmann showed in \cite{ChV1} that for any RAAG $\AG$, $Out(\AG)$ is virtually torsion-free and has finite virtual cohomological dimension, and they showed  in \cite{ChV2} that $Out(\AG)$ is residually finite; this last result was also proved by different methods by Minasyan \cite{Minasyan}. Also Charney and Vogtmann showed in \cite{ChV2} that for $\G$ a homogeneous defining graph, $Out(\AG)$ satisfes the Tits' alternative: every subgroup of $Out(\AG)$ is either virtually solvable or contains a non-abelian free group.

Many results concerning $Out(F_n)$ were obtained through the study of Outer Space, a contractible space on which $Out(F_n)$ acts with finite point stabilizers.

Outer Space, which we denote by $O(F_n)$, was introduced by Culler and Vogtmann \cite{CV}. The points of $O(F_n)$ are given by actions of $F_n$ on trees. The trees are simplicial $\R$-trees, and can be described as trees that have been equipped with a length metric by assigning one of finitely many lengths to each edge. The actions are simplicial, free, minimal actions by isometries, where a minimal action is one under which there is no invariant subtree.  

For each action $\alpha: F_n \circlearrowleft T$, there is an associated translation length function $l^\alpha$. This is the function $l^\alpha:F_n \rightarrow \R$  defined by $l^\alpha(g)=\inf_{x\in T}\{d(x,g.x)\}$; the element $g\in F_n$ is sent to the infimum of distances moved by points of $T$ under the action of $g$. The assumptions ensure that the infimum is realized as a minimum.

The function $l^\alpha$ is constant on conjugacy classes. Let $\Omega$ be the set of conjugacy classes of $F_n$ and consider the map $l^\alpha: \Omega \rightarrow \R$. Note that $\Omega$ is countable, since $F_n$ is countable. One considers $l^\alpha$ to be an element of 
$\R^{\Omega}=$
$\R^\infty$; then up to scaling, we get an element of $\bbP^\infty$, 
called a projectivized length function. 

Thus there is a map from the set of actions $\alpha$ to the space $\bbP^\infty$, defined by sending each action to its corresponding projectivized translation length function. If we consider trees up to scaling of their metrics, then we get the map 
$$\left\{
\begin{array}{c l}   
    \text{minimal actions of }F_n\\
    \text{on trees, up to scaling}
\end{array}\right\}\longrightarrow \bbP^\infty.$$

Culler and Morgan \cite{CM} proved two important theorems concerning the relationship between actions on trees and their associated  length functions. The first theorem shows that the  length function of a minimal action of a free group on a tree completely determines the action. Thus the map 
$$\left\{
\begin{array}{c l}   
    \text{minimal actions of }F_n\\
    \text{on trees, up to scaling}
\end{array}\right\}\longrightarrow \bbP^\infty$$
is an injection.
In fact, Culler and Vogtmann give $O(F_n)$ the topology it inherits as a subspace of $\bbP^\infty$. It turns out that this subspace topology is the same as the equivariant Gromov-Hausdorff topology of a set of actions on metric spaces (see \cite{Best,Paulin}).

The second theorem says that the image of this map lies within a compact subspace of $\bbP^\infty$ and hence has compact closure in $\bbP^\infty$. In relation to Outer Space, the first theorem shows that Outer Space can be embedded in $\bbP^\infty$, while the second theorem gives a compactification of Outer Space. 

The points of $O(F_n)$ are actions of $F_n$ on trees; this is reasonable since free groups are the fundamental groups of graphs and thus act naturally on their universal covers, which are trees. Homogeneous  RAAGs of dimension 2 are the fundamental groups of cell complexes formed by gluing together 2-dimensional tori. Thus they act naturally on the universal covers of such complexes; these universal covers are 2-dimensional CAT(0) rectangle complexes. 

Recently Charney, Stambaugh, and Vogtmann have constructed analogues of Outer Space for RAAGs. The points of these Outer Spaces are given by minimal actions on CAT(0) rectangle complexes $X$.

We work in the setting of homogeneous RAAGs $\AG$ of dimension 2. In this case, the definition of a minimal action is more technical than that in the free group case. First of all, a minimal action must be simplicial, proper, semisimple, and by isometries, where an action $\alpha$ is \emph{semisimple} if $l^\alpha(g)=\inf_{x\in X}d(x,g.x)$ is realized as a minimum. The {\sl minset} of $g\in \AG$ is $\text{Min}^\alpha(g)= \{x \in X : d(x, g.x)= l^\alpha(g)\}$, and the minset of a subgroup $H\subset \AG$ is $\cap_{h\in H}\m(h)$. Finally, a \emph{minimal} action is a  simplicial, proper, semisimple action by isometries that also has the property that $X$ is covered by minsets of $\Z^2$-subgroups.

A theorem of Charney and Margolis \cite{ChM}, analogous to the first theorem of Culler and Morgan, says that the length function of a minimal action completely determines the action. Thus, analogous to the free group case, 
\[O(\AG):=\left\{
\begin{array}{c l}   
    \text{minimal actions of }\AG \text{ on 2-}\\
    \text{dimensional CAT(0) rectangle}\\
    \text{complexes, up to scaling}
\end{array}\right\}\longrightarrow \bbP^\infty\] is an injection and $O(\AG)$ can be embedded in $\bbP^\infty$. 
In fact, a topology can be put on $O(\AG)$ via this map. 

To get a compactification of $O(\AG)$, one need only show that the image of this map has compact closure. Our main theorem does exactly this; it is an analogue of Culler and Morgan's second theorem.

\begin{reptheorem}{main}  Let $\AG$ be a homogeneous RAAG of dimension 2. Let $O(\AG)$ be the set of minimal actions of $\AG$ on 2-dimensional CAT(0) rectangle complexes, up to scaling. Under the map that sends an action to its projectivized translation length function, the image of $O(\AG)$ lies in a compact subspace of $\bbP^\infty$.
\end{reptheorem}

The proof of Culler and Morgan's second theorem depended on a key proposition. The analogue of that proposition is our Proposition~\ref{key}. Define the cyclically reduced word length $\|g\|$ of an element $g\in \AG$ to be the length of the shortest word in the generators of $\AG$ or their inverses that represents an element conjugate to $g$.

\begin{repproposition}{key} Let $\G$ be a connected, triangle-free graph, and let $\AG$ be the associated RAAG. There exists a finite subset $D\subset \AG$ and a constant $C$ such that for any minimal action $\alpha$ of $\AG$ on a 2-dimensional CAT(0) rectangle complex, and for any $g\in \AG$, $l^\alpha(g) \leq CM^\alpha \| g\|$, where $M^\alpha=\displaystyle\max_{d\in D}l^\alpha(d)$.
\end{repproposition}

In fact, our Theorem~\ref{main} follows from Proposition~\ref{key} in essentially the same way that Culler and Morgan's theorem did.

However, our proof of  Proposition~\ref{key} differs completely from that of Culler and Morgan's. 
Their proof was by induction on $\|g\|$ and used the following property satisfied by the length function of any isometric action $\alpha$ of a group $G$ on a tree: for all $g,h \in G$, either $l^\alpha(gh)=l^\alpha(gh^{-1})$ or $\max(l^\alpha(gh), l^\alpha(gh^{-1})) \leq l^\alpha(g) +l^\alpha(h)$. 
This property is not satisfied by our actions, so we were forced to take a different tack. 

Our proof depends on the geometry of CAT(0) rectangle complexes.

In particular, it depends on understanding the relationship between translation lengths and distances between minsets. More specifically, it depends on being able, for a given action, to bound the distance between two minsets of generators by the translation length of a product of the generators. 
In fact, we work with $d_1$ distances (based on lengths of edge paths) and $l_1$ translation lengths, rather than CAT(0) distances and CAT(0) translation lengths, and develop rules for determining which edge paths can be used to measure $l_1$ translation lengths.

The proof also depends on being able, for a given action, to find a base point whose distance from all minsets of generators is bounded by a function of the distances between the minsets. Here we again work with $d_1$ distances and we make use of some CAT(0) geometry, for example, the fact that each minset is isometric to a $\text{tree}\times \R$, a consequence of Theorem II.6.8(4) in \cite{BH}.

These geometric arguments form the core of the proof of Proposition~\ref{key} and thus of Theorem~\ref{main}.


\section{Preliminaries}

In this thesis, we work with CAT(0) rectangle complexes. Before introducing these, we first review the definition of CAT(0) spaces and give a few useful theorems concerning them. 

Given a metric space, a \emph{geodesic} is an isometric embedding of a line segment (or line) into the space. Since such a path has length the distance between the points, a geodesic is also a path of minimal length between its endpoints. 
A \emph{geodesic metric space} is a metric space for which any two points of the space are joined by a geodesic. 
Note that a geodesic metric space is necessarily a \emph{length space}: a metric space for which the distance between any two points is the infimum of the lengths of the paths that join them. 

Now let $X$ be a geodesic metric space; we will say what it means for $X$ to be CAT(0). 
First note that a triangle of $X$ is specified by three points along with three geodesics joining the three pairs of points. Given any triangle in $X$, we can form a comparison triangle in Euclidean 2-space whose side lengths are the same as those of the three geodesic segments of $X$, and for each point of the triangle, there is a corresponding point of the comparison triangle, specified by its distance from the endpoints of the geodesic segment on which it lies.  Informally, $X$ is CAT(0) if triangles in $X$ are at least as thin as their comparison triangles. Formally, $X$ is CAT(0) if for any two points of any triangle of $X$, the distance between them is less than or equal to the distance (in Euclidean 2-space) between the corresponding two points of the comparison triangle.

A space that is locally CAT(0) is said to be non-positively curved. 
Note that hyperbolic space and Euclidean space are both CAT(0), and in fact any manifold of non-positive sectional curvature will be non-positively curved in the CAT(0) sense. 
Thus the CAT(0) condition applies to a larger class of spaces; still many results can be derived from the assumption of CAT(0). 

For example, CAT(0) spaces are uniquely geodesic, that is, given any two points, there is a \emph{unique} geodesic segment joining them. We often denote the unique geodesic that joins two points $a$ and $b$ by $[a,b]$. Furthermore, geodesics vary continuously with their endpoints. From this it follows that every open ball is contractible, and in fact that the CAT(0) space itself is contractible. This property that CAT(0) spaces are contractible is an important one. 
Also, in a CAT(0) space, every local geodesic is a (global) geodesic, and open balls are convex. (See Bridson and Haefliger \cite{BH} for details.) 

We've seen above that every CAT(0) space is non-positively curved, is contractible, thus simply connected, and is a geodesic space, thus a length space. An important theorem is the following.

\begin{theorem}
(Cartan-Hadamard; Theorem II.4.1 in \cite{BH}) Any complete length space that is both non-positively curved and simply connected is (globally) CAT(0).
\end{theorem}

This theorem is often called a `local-to-global' theorem, since it says that if a simply connected complete length space is locally CAT(0), then it is globally CAT(0).

One theorem that we rely on (albeit indirectly) is the Flat Quadrilateral Theorem \cite{BH}. Before stating the theorem, we give some necessary definitions.

First, we define the Alexandrov angle between two geodesics of a geodesic space $X$. Suppose $c:[0,a] \rightarrow X$ and $c^\prime: [0,b] \rightarrow X$ are geodesics with the same starting point $c(0)=c^\prime(0)$. We define the angle (at $c(0)$) between $c$ and $c^\prime$ using the limits of angles in comparison triangles. Specifically, if $\overline{\triangle}(c(0),c(t), c^\prime(t^\prime))$ is the comparison triangle for the triangle with vertices $c(0)$, $c(t)$, and $c^\prime(t^\prime)$, and $\overline{\angle}_{c(0)}(c(t), c^\prime(t^\prime))$ is the angle of 
$\overline{\triangle}(c(0),c(t), c^\prime(t^\prime))$ 
at the point corresponding to $c(0)$, then the angle between $c$ and $c^\prime$ is $lim_{t, t^\prime \rightarrow 0}\overline{\angle}_{c(0)}(c(t), c^\prime(t^\prime))$. 

Second, two quick definitions: first, a subspace $A$ of a geodesic metric space is \emph{convex} if for any two points of $A$, any geodesic that connects them is contained in $A$, and second, the \emph{convex hull} of a subset $B$ of a geodesic space is the intersection of all convex subspaces that contain $B$. We note for later use that a convex subspace of a CAT(0) space is itself a CAT(0) space.

\begin{theorem}(Flat Quadrilateral Theorem; Theorem II.2.11 in \cite{BH})
Consider four points $p$, $q$, $r$, and $s$ in a CAT(0) space $X$, along with the four geodesic segments $[p,q], [q, r], [r, s]$, and $[s,p]$.
If the angle between the segments at each corner is at least $\frac{\pi}{2}$, then the convex hull of $p$, $q$, $r$, and $s$ is isometric to a Euclidean rectangle. 
\end{theorem}

This thesis is concerned with isometric actions on certain kinds of CAT(0) spaces. We now discuss isometric actions on an arbitrary CAT(0) space.

An action by a group $G$ on a metric space $X$ is a map $G\rightarrow Homeo(X)$, where $Homeo(X)$ is the group of homeomorphisms of $X$. 
The action is said to be \emph{isometric} if the image lies in the group of isometries of $X$. 
The action is said to be \emph{cocompact} if there is a compact set $K$ such that $X=G.K$. 
Finally, we say that the action is \emph{proper} if each element $x \in X$ has an open neighborhood $N(x)$ for which the set $\{g\in G: N(x) \cap g. N(x) \neq \emptyset\}$ is finite. 
Note that an action of a group $G$ on a topological space $X$ is often said to be proper if for every compact subset $K$ of $X$, the set $\{g\in G: K \cap g. K \neq \emptyset\}$ is finite. The two definitions are equivalent if $X$ is a proper metric space (that is, closed balls are compact) and $G$ acts by isometries on $X$ (see Bridson and Haefliger \cite{BH}); thus the two definitions are equivalent for the actions we consider.

Every isometric action $\alpha$ of a group $G$ on a metric space $X$ has an associated \emph{translation length function} $l^\alpha$. This is the function $l^\alpha:G \rightarrow \R$ that sends an element $g \in G$ to $l^\alpha(g):=inf_{x\in X}\{d(x,g.x)\}$; that is, it sends $g$ to the infimum of distances moved by points of $X$ under the action of $g$.

For each action $\alpha:G \circlearrowleft X$ and each element $g \in G$, we define the {\sl minset} of  $g$ by $\text{Min}^\alpha(g):= \{x \in X : d(x, g.x\})=l^\alpha(g)$; thus $\m(g)$ is the set of points of $X$ for which the distance moved under the action of $g$ attains its infimum. Also, for each subgroup $H \subset G$, we define  $\m(H)=\cap_{h\in H}\m(h)$.

An element $g\in G$ is said to be \emph{semisimple} if $\m(g)$ is non-empty; that is, if the infimum $l^\alpha(g)$ of distances moved under the action of $g$ is attained. Also, an element $g$ is said to be \emph{hyperbolic} if it is semisimple and $l^\alpha(g) > 0$. An action $\alpha:G\circlearrowleft X$ is said to be \emph{semisimple} if every $g\in G$ is semisimple, and \emph{hyperbolic} if every  non-identity element $g\in G$ is hyperbolic.

We mention some properties of  length functions and minsets that are important for this thesis; these and others are discussed in Bridson and Haefliger \cite{BH}. For an isometric action $\alpha: G \circlearrowleft X$, for any $g,h \in G$,  $l^\alpha (g)= l^\alpha(hgh^{-1})$. Thus the length function $l^\alpha$ is constant on conjugacy classes of $G$. Furthermore, $h.\m(g)=\m(hgh^{-1})$. In particular, if $h$ commutes with $g$, then $h.\m(g)=\m(g)$; for example, $g.\m(g)=\m(g)$. 

If $X$ is a CAT(0) space, then for any $g\in G$, $\m(g)$ is convex. If $g$ is a hyperbolic isometry, then we know quite a lot about the geometry of $\m(g)$. In fact, $\m(g)\cong Y \times U$, where $U \cong \R$ and $Y$ is a convex subspace of $X$ (throughout this thesis we use $\cong$ to denote that two spaces are isometric). Furthermore the action of $g$ on $\m(g)$ sends a point $(y, t)$ to the point $(y, t+ l^\alpha(g))$. In other words, $g$ acts by the identity on the component $Y$ and by translation by $l^\alpha(g)$ on the component $U\cong \R$. Accordingly, $\m(g)$ is a union of axes of $g$, where an axis of $g$ is a geodesic line on which $g$ acts by non-trivial translation. 

In fact, we can say more than this. The following theorem is important to this work.

\begin{theorem}
(The Flat Torus Theorem; Theorem II.7.1 in \cite{BH}) Suppose the group $H\cong \Z^n$ acts properly, by semi-simple isometries on the CAT(0) space $X$. Then $\m(H)\cong Y \times \E^n$, where $Y$ is a convex subspace of $X$. Furthermore, every $h\in H$ acts as the identity on the component $Y$ and as translation by $l^\alpha(h)$ on the component $\E^n$. Finally, the quotient of each $n$-flat $\{y\} \times \E^n$ by the action of $H$ is an $n$-torus.
\end{theorem}

We will shortly return to this theorem to talk about some implications, but first we discuss cubical complexes and rectangle complexes.

A \emph{cubical complex} is a cell complex $X$ where each $n$-cell $\sigma$ is a unit Euclidean cube with injective characteristic map $p: \sigma \rightarrow X$, and if the images of two characteristic maps $p_1$,~$p_2$ intersect in the set $R\neq \emptyset$, then $p_1^{-1}(R)$,~${p_2}^{-1}(R)$ are faces of $\sigma_1$,~$\sigma_2$, and ${p_2}^{-1}\circ p_1$ is an isometry between these faces  (see \cite{BH}). $X$ is equipped with the length metric $d$ induced by the Euclidean metric on each cell. That is, for any two points $x_1, x_2 \in X$, $d(x_1, x_2)$ is defined to be the infimum of the lengths of paths joining $x_1$ and $x_2$ (where the length of a path is defined using the lengths of subpaths that are contained in single cells.) We think of cubical complexes as being constructed by gluing together unit Euclidean cubes along their faces. Note that cubical complexes and simplicial complexes both satisfy the same kinds of additional conditions on the basic cell complex construction.

Cubed complexes are defined similarly but with fewer constraints, so that every cubical complex is a cubed complex. By definition, a \emph{cubed complex} is a cell complex $X$ where each $n$-cell $\sigma$ is a unit Euclidean cube whose characteristic map restricts to an isometry on each face (see \cite{BH}); the metric is the length metric induced by the Euclidean metric on each cell. In contrast to cubical complexes, characteristic maps are not required to be injective and two cubes may intersect in more than one face. An example of a cubed complex that is not a cubical complex is the 2-dimensional torus, constructed as usual using one vertex, two edges, and a square. 

In this thesis, the spaces we deal with are rectangle complexes. These are described  in exactly the same way as cubical complexes, except that each cell is a \emph{hyperrectangle}: a finite product  $\prod[0,r_i]$ of intervals of possibly different lengths; in particular a 2-cell is a Euclidean rectangle. Different hyperrectangles of the same dimension may have different edge lengths. However, in this thesis, each rectangle complex is assumed to have only finitely many isometry types of cells (this property is denoted by saying the complex \emph{has finite shapes}), so it has only finitely many different edge lengths. We think of rectangle complexes as being constructed by gluing together Euclidean hyperrectangles along their faces. Any rectangle complex with finite shapes is a complete geodesic metric space, by a theorem of Bridson (see \cite{BH}). 

We say that a rectangle complex is \emph{homogeneous of dimension $n$} 
if every maximal hyperrectangle is $n$-dimensional. Our main theorem concerns actions on rectangle complexes that are homogeneous of dimension 2; we think of such complexes as being constructed by gluing together Euclidean (2-dimensional) rectangles along their edges.

The \emph{link} $lk(x)$ of a vertex $x$ of a rectangle complex $X$ is the simplicial complex that contains one $k$-dimensional  simplex for every $(k+1)$-dimensional hyperrectangle containing $x$. We can think of $lk(x)$ as the set of points of $X$ of distance $\epsilon$ from $x$, where $\epsilon$ is smaller than the length of every edge with endpoint $x$, endowed with the natural simplicial complex structure inherited from the rectangle complex structure of $X$. For example, if $X$ is composed of two (2-dimensional) rectangles joined along an edge, and $x$ is an endpoint of that edge, then $lk(x)$ is a graph consisting of two edges joined at a single vertex.

The link of a vertex of a cubed complex is defined in the same way, except that the link might not be a simplicial complex: in particular, in the cell complex structure of the link, characteristic maps might not be injective, and two simplices might meet in more than a single face.

A \emph{flag complex} is a simplicial complex for which every complete subgraph of the 1-skeleton spans a simplex. If $x$ is a vertex of a cubed complex $X$, and if $lk(x)$ is a simplicial complex, then it is a flag complex if and only if every collection of $(n-1)$-dimensional cubes of $X$ that could possibly bound a corner of an $n$-dimensional cube at $x$ actually does so. We have the following theorem, due to Gromov. 

\begin{theorem}
\label{npc}
(Theorem II.5.20 in \cite{BH}) 
A finite-dimensional cubed complex is non-positively curved if and only if the link of each vertex is a flag complex.
\end{theorem}
 
Recall that the torus is a cubed complex (though not a cubical complex). The unique vertex of the torus has link a cycle graph with four edges. This is a simplicial complex since it contains no loops or bigons, and it is flag since it contains no triangles. Thus by Theorem~\ref{npc} the torus is non-positively curved.

Since the link of a point in a rectangle complex does not change if we rescale all edges to have length one, a version of Theorem~\ref{npc} holds for rectangle complexes as well: a rectangle complex with finite shapes has non-positive curvature if and only if the link of each vertex is a flag complex. In the case of a 2-dimensional rectangle complex, the link of a vertex is a simplicial graph; thus it is a flag complex if and only if it contains no triangles. 

Combining the rectangle complex version of Theorem~\ref{npc} with the Cartan-Hadamard Theorem, we get the following.  Note that the assumptions of the Cartan-Hadamard Theorem hold for a rectangle complex of finite shapes, since as noted above, such a complex is a complete geodesic space and thus a complete length space.

\begin{theorem}
\label{rc CAT(0)} 
A rectangle complex with finite shapes is a CAT(0) space if and only if it is simply connected and the link of each vertex is flag.
\end{theorem}

CAT(0) cubical complexes have been well-studied, and many properties also hold for CAT(0) rectangle complexes. We mention a few useful results here (see \cite{S} for details).

An important concept is that of a wall. First, we give a preliminary definition: for any $n$-dimensional hyperrectangle $\sigma$, a \emph{midplane} is an $(n-1)$-dimensional hyperrectangle parallel to one of the $(n-1)$-dimensional faces $F$ of $\sigma$ and passing through the midpoint of any edge perpendicular to $F$. In a rectangle complex, two midplanes are equivalent if they intersect in the midpoint of an edge. This generates an equivalence relation on midplanes; a \emph{wall} is the union of the midplanes of an equivalence class.

If the rectangle complex $X$ is CAT(0), walls have many nice properties. Suppose $W$ is a wall of a CAT(0) rectangle complex. Then (1) the union of hyperrectangles that contain a midplane of $W$ is a convex subcomplex of $X$, (2) $W$ is convex and thus itself a CAT(0) space (in fact, it inherits the structure of a CAT(0) rectangle complex), (3) $X\setminus W$ has two connected components, and (4) every collection of pairwise intersecting hyperplanes has non-empty intersection. Also, a geodesic that intersects $W$ in a segment of positive length must be completely contained in $W$. Since $W$ is convex, this implies that any geodesic that does not lie entirely in $W$ crosses $W$ at most once.

Let $A_\Gamma$ be a RAAG that is homogeneous of dimension $n$. In this thesis, we consider isometric, proper, semisimple actions by $\AG$ on $n$-dimensional CAT(0) rectangle complexes $X$. Suppose we have such an action by a RAAG $\AG$ on the $n$-dimensional CAT(0) rectangle complex $X$. Properness of the action along with the fact that $\AG$ is torsion-free implies that the action is \emph{free}: every point of $X$ is moved by every non-identity element of $\AG$. But then the action is free and proper; thus each element $x \in X$ has an open neighborhood $N(x)$ for which all translates $g.N(x)$ (where $g\in \AG$) are disjoint, and so the quotient map $X\rightarrow X/\AG$ is a covering map. Since $X$ is assumed to be CAT(0), it is contractible and thus simply connected; also the open balls around any point $x$ are convex, so $X$ is locally path-connected. Thus $\AG \cong \pi_1(X/\AG)$. 
Finally note that since the action is free and semisimple, every non-identity element of $\AG$ acts as a hyperbolic isometry-- so the minset of every non-identity element of $\AG$ has the geometry discussed above. 

Now suppose that $X$ is $n$-dimensional. If $H\cong \Z^n$ is a subgroup of $\AG$, we say that $H$ is a $\Z^n$-subgroup; like $\AG$, $H$ must act properly, by semi-simple isometries on $X$. By the Flat Torus Theorem, then, $\m(H) \cong Y \times \E^n$, where $Y$ is a convex subspace of $X$. Since $X$ is $n$-dimensional, $Y$ must consist of a single point, so $\m(H) \cong \E^n$. 

Summarizing the last two paragraphs: an isometric, proper, semisimple action by a RAAG $\AG$ on an $n$-dimensional CAT(0) rectangle complex $X$ will also be free and every non-identity element $g$ will act as a hyperbolic isometry. Also the map $X\rightarrow X / \AG$ is a covering map, with $\AG \cong \pi_1(X/\AG)$. Finally, minsets of $\Z^n$-subgroups are isometric to Euclidean $n$-flats. 

We now define a minimality condition, the final condition that we will require our actions to satisfy.

\begin{definition}
Suppose $A_\Gamma$ is homogeneous of dimension $n$. 
An isometric, proper, semi-simple action by $\AG$ on an $n$-dimensional CAT(0) rectangle complex $X$ is \emph{minimal} if every point of $X$ lies in the minset of some $\Z^n$-subgroup of $G$.
\end{definition}

This minimality assumption was introduced by Charney and Margolis \cite{ChM} in the 2-dimensional case as a sufficient condition for showing that the length function of an action on $X$ completely determines the isometry type of $X$ and in fact the action itself. In a sense, the assumption of minimality rids our space of extraneous points, that is, points unimportant for the structure of the action.

To get a better feel for this, let's think about the 1-dimensional case. A 1-dimensional CAT(0) rectangle complex is a metric tree. For an isometric, proper, semi-simple action by $\AG$ on a tree $X$, the minset $\m(g)$ of any element $g\in \AG$ is isometric to a real line, called the axis of $g$. Thus our definition of minimality says that an action is minimal if every point of $X$ lies in the axis of some $g\in \AG$. 

We now describe an example of an action that is not minimal. Consider a tree $X$, all of whose edges have length 1, consisting of a real line plus added edges: in particular, a single edge, called a leaf, is attached to each vertex of the line. Suppose the cyclic group $G=\langle g \rangle$ acts on $X$ by translating the line by distance 1. This action is isometric, proper, and semi-simple, but the leaves do not lie in the minset of any group element, so the action is not minimal. Again, the idea is that the attached edges and leaves are inessential for the description of the action. Note that the line forms an invariant subtree under the given action by $G$. 
In fact, in the 1-dimensional case, our minimality condition is equivalent to requiring that under the action, $X$ has no invariant subtree.

We return to the Flat Torus Theorem. Assume that a $\Z^n$-subgroup $H$ of a RAAG $\AG$ acts properly, semi-simply by simplicial isometries on an $n$-dimensional CAT(0) rectangle complex $X$. Recall that the Flat Torus Theorem tells us that $\m(H) \cong \E^n$. Since $\m(H)\cong \E^n$ has the same dimension as a maximal hyperrectangle of $X$, $\m(H)$ is a subcomplex, with a rectangle structure inherited from $X$. Note that $H$ acts simplicially, so it must preserve the rectangle structure on $\m(H)$. By the Flat Torus Theorem, $H$ acts by translations on $\m(H)$ and the quotient of $\m(H)$ by this action is an $n$-torus. Hence $H$ is a finite index subgroup of the group of translations that preserve the rectangle complex structure, so $H$ contains elements that translate parallel to each of the $n$ grid directions. Such an element is called a \emph{gridline isometry}.

We also return to the Flat Quadrilateral Theorem to discuss one implication. Let $A$, $B$ be convex subcomplexes of a CAT(0) rectangle complex $X$. A \emph{spanning geodesic} between $A$ and $B$ is a minimal length geodesic joining $A$ to $B$. It is noted in \cite{ChM} that if $X$ has finite shapes, a spanning geodesic always exists. Also, any spanning geodesic will form angle $\geq \frac{\pi}{2}$ with both $A$ and $B$ (otherwise, a shorter path exists). Now suppose there are two spanning geodesics, one with endpoints $a_1$ and $b_1$ and the other with different endpoints $a_2$ and $b_2$. Then $[a_1, b_1]$, $[b_1, b_2]$, $[b_2, a_2]$, and $[a_2, a_1]$ form a quadrilateral with all four angles greater than or equal to $\frac{\pi}{2}$. Thus by the Flat Quadrilateral Theorem, the convex hull of $a_1, b_1, b_2$, and $a_2$ is isometric to a Euclidean rectangle. This implies that the \emph{bridge} between $A$ and $B$, that is, the union of the spanning geodesics, is a convex subspace of $X$.

Finally, we describe, for any RAAG $\AG$, an isometric, proper, semisimple action on a canonical CAT(0) rectangle complex $X$. First, we construct the Salvetti complex $S_\G$ for a given RAAG $\AG$, where $\AG$ has generators $v_1, \dots, v_n$. The 1-skeleton  is a rose: it has one vertex and $n$ edges, each of length one, labeled by  $v_1, \dots, v_n$. For each commutator relation $v_1v_2=v_2v_1$, we add a unit square 2-cell, attached along the $v_1$ and $v_2$ edges to form a torus. For every three generators that pairwise commute, we add a 3-torus, attached along the three 2-tori associated to the three generators. We continue in the same fashion, adding a $k$-torus for every $k$ generators that pairwise commute. This is clearly a finite process. The resulting finite-dimensional cubed complex $S_\G$ is called the Salvetti complex for $\AG$. Note that by construction, the link of the unique vertex is a flag complex, so $S_\G$ is non-positively curved.

Using Theorem~\ref{rc CAT(0)}, we see that the universal cover $\tilde{S}_\G$  is a CAT(0) rectangle complex. Note that by construction, $\pi_1(S_\G) \cong \AG$, so $\AG$ is the group of deck transformations of $\tilde{S}_\G$. The action $\AG \circlearrowleft \tilde{S}_\G$ is clearly isometric, proper, and semi-simple. Thus we see that RAAGs act naturally on CAT(0) rectangle complexes. 

With regard to RAAGs, we mention word lengths. Let $\AG$ be a RAAG with generating set $V=\{v_1, \dots, v_n\}$.  Clearly any element of $\AG$ can be expressed as a word in letters from the set $V^{\pm1}=\{v^{\pm1}_1, \dots, v^{\pm1}_n \}$. The \emph{support} of a word $W$ is the set of all generators $v_i$ such that $v^{\pm1}_i$ appears as a letter in $W$. A \emph{reduced word} for $g\in \AG$ is a minimal length word representing $g$; it turns out that any two reduced words for $g$ have the same support. The \emph{reduced word length} $|g|$ of $g$ is the length of a minimal length word representing $g$. For example, if $\AG=\langle u, v, w : uv=vu \rangle$ then 
$|wvw^{-1}|=3$ and $|uvu^{-1}|=1$.

Also, an element $g\in \AG$ is said to be \emph{cyclically reduced} if it cannot be expressed as $g=aha^{-1}$ with $a\in V^{\pm 1}$ and $|h| < |g|$. For every element $g$, there is a unique cyclically reduced element $h$ such that $g=bhb^{-1}$ for some $b\in \AG$ where $|g|=|h|+2|b|$. The \emph{cyclically reduced word length} $\|g\|$ of $g$ is the reduced word length $|h|$ of this cyclically reduced element $h$. For example, for $\AG$ as above, $\|wvw^{-1}\|=\|uvu^{-1}\|=1$. All reduced words representing cyclically reduced elements of the same conjugacy class are cyclic permutations of each other; thus they have the same support and same length. In particular, the elements of a conjugacy class have the same cyclically reduced word length.


\section{Results}

Let $X$ be a CAT(0) rectangle complex. For two connected subsets $A$ and $B$ of $X$, we will say that a wall \emph{separates} $A$ and $B$ if $A$ and $B$ lie entirely in different components of $X \setminus W$.

Define the width $w(W)$ of a wall $W$ to be the length of any edge that crosses $W$; this is well-defined since any two edges that cross $W$ are opposite sides of some rectangle in $X$. Note that then the length $l(P)$ of an edge path $P$ is the sum of the widths of the walls that cross $P$, counted with multiplicity. 

For future reference, we note that we will use $l(Q)$ to denote the length of a path $Q$, whether $Q$ is a CAT(0) geodesic or an edge path.

Now consider the collection of paths that have one endpoint in $A$ and the other in $B$, where $A$ and $B$ are convex subcomplexes. Clearly each such path must cross at least the walls that separate $A$ from $B$. Thus if there exists a path that crosses \emph{only} the separating walls, and crosses each of those \emph{exactly once}, it will necessarily have minimal length among paths from $A$ to $B$.

We wish to show a converse: we want to show that any minimal length path from $A$ to $B$ will cross \emph{only} the separating walls, and will cross each of those \emph{exactly once}. (Note that this will show that a path that crosses only the separating walls, each exactly once, actually exists.)

This is not hard in the case that $A$ and $B$ are single vertices $a$ and $b$. In fact Theorem 4.13 of \cite{S} tells us that any minimal length path between vertices will cross no wall more than once. Thus a minimal path between vertices can cross only separating walls, and each must be crossed exactly once.

We use edge paths to define the $d_1$ metric on the set of vertices of $X$: for any two vertices $a, b$ of $X$, $d_1(a, b)$ is the length of the shortest edge path from $a$ to $b$. Thus by the previous paragraph, $d_1(a, b)=\sum_{W\in S_{a,b}}w(W)$, where $S_{a,b}$ is the set of walls that separate $a$ from $b$. 

To show the analogous results for convex subcomplexes $A$ and $B$, we must do some more work.

In what follows, we use the concept of the cellular neighborhood $N(W)$ of a wall $W$,
the union of those rectangular cells that have a midplane in $W$. Note that if $w(W)$ is the width of wall $W$, then $N(W)\cong W \times [0,w(W)]$  is convex, and so is $W \times \{t\} \subset N(W)$, for any fixed $t \in [0,w(W)]$ \cite{S}. 

The next lemma is a technical one which will be useful in proving what follows.

\begin{lemma}
\label{wall subcomplex edge}
 Suppose wall $W$ intersects the convex subcomplex $B$, and suppose $W$ crosses edge $e$ which has endpoint $b \in B$. Then edge $e$ lies in $B$.
\end{lemma}

\begin{proof}
Let $y$ be a point in the intersection of wall $W$ with subcomplex $B$ (see Figure~\ref{fig:2}). 
\begin{figure}[htbp]
  \centering
  \includegraphics[width=3in]{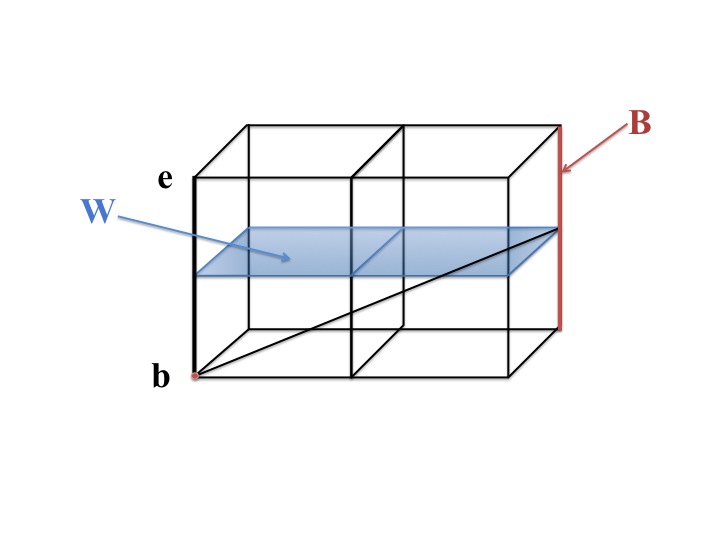}
  \caption{Illustration for Lemma~\ref{wall subcomplex edge}.}
  \label{fig:2}
\end{figure}

The geodesic $[b,y]$ lies in $B$, since $B$ is convex. Also $[b,y]$ passes through the interior of some rectangular cell containing $e$, because any point of $(b,y]$ has non-zero second coordinate as a point of $N(W)\cong W \times [0,w(W)]$ while $e=\{b\}\times [0,w(W)]$. Since $B$ is a subcomplex, this means that $e$ lies in $B$.
\end{proof}

The following lemma basically says that if a subcomplex is convex (in the CAT(0) metric), then it is convex in the $d_1$ metric.

\begin{lemma}
\label{edge path convex}
For any two vertices of a convex subcomplex $A$ of $X$, \emph{every} minimal edge path joining them is contained in $A$. 
\end{lemma}

\begin{proof} 
Take a minimal length edge path $P$ between vertices $a_1$ and $a_2$. By our discussion above, since $P$ is minimal, every wall that crosses $P$ must separate $a_1$ from $a_2$. But the geodesic $[a_1, a_2]$ lies in $A$ (by the convexity of $A$) and crosses every wall that separates $a_1$ from $a_2$. So every wall that crosses $P$ must intersect $A$.

We want to show that our minimal edge path $P$ never leaves $A$. Suppose edge $e$ of $P$ has endpoint $x \in A$ and crosses wall $W$. By the above, $W$ intersects $A$. Then by Lemma~\ref{wall subcomplex edge}, $e$ lies in $A$. So path $P$ never leaves $A$, that is, $P$ is entirely contained in $A$. 
\end{proof}

\begin{corollary}
\label{boundary N(W)}
Any minimal edge path that does not cross wall $W$ but that joins two edges that cross $W$ must lie entirely in one of the boundary components of $N(W)$, that is, in either $W\times \{0\}$ or $W\times \{w(W)\} \subset N(W)$.
\end{corollary}

Note that Corollary~\ref{boundary N(W)} is also contained in Theorem 4.13 of \cite{S}.

\begin{lemma}
\label{min2}
Let $A$ and $B$ be convex subcomplexes of $X$. Then a minimal length edge path between $A$ and $B$ crosses only walls that separate $A$ from $B$, each exactly once.
\end{lemma}

\begin{proof}
Let $P$ be a minimal length edge path from $A$ to $B$, and say the endpoints of $P$ are $a \in A$ and $b \in B$.

Since $P$ is minimal between $a$ and $b$, each wall that it crosses is crossed exactly once, by Theorem 4.13 of \cite{S}. 
Furthermore $P$ must cross each wall that separates $A$ from $B$. To finish the proof, we need only show that $P$ crosses \emph{only} walls that separate $A$ from $B$.

Suppose for contradiction that $P$ crosses some wall $W$ that does not separate $A$ from $B$. Since $P$ crosses $W$ only once, and yet $W$ is not separating, it must be the case that $W$ intersects either $A$ or $B$; let's say $B$. If there is more than one such wall, choose $W$ to be the one whose intersection with $P$ is closest to $b$, let $e_1$ be the edge of $P$ that crosses $W$, and let $P^\prime$ be the subpath of $P$ that connects $e_1$ and $b$ (see Figure~\ref{fig:3}). 
\begin{figure}[htbp]
  \centering
  \includegraphics[width=3in]{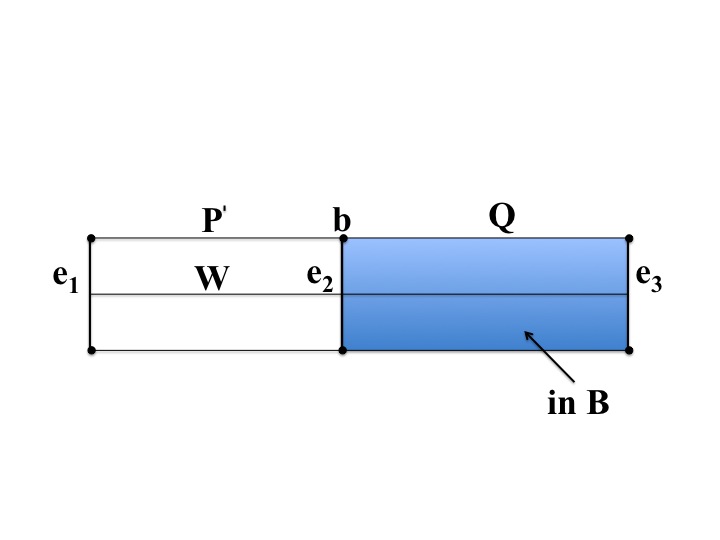}
  \caption{Illustration for Lemma~\ref{min2}.}
  \label{fig:3}
\end{figure}
Of all the edges in $B$ that cross $W$, let $e_3$ be one at a minimal distance from $b$ and take a minimal edge path $Q$ from $b$ to $e_3$ (note that path $Q$ does not cross wall $W$). We'll see in a moment that in fact then $b$ is an endpoint of $e_3$.

Note that by choice of $P^\prime$ and $Q$, path $P^\prime Q$ does not cross wall $W$.

Also, path $P^\prime Q$ is minimal, since it crosses no wall more than once. This is because no wall intersects $P^\prime$ more than once (by assumption about $P$), no wall intersects $Q$ more than once (since $Q$ is a minimal edge path between its endpoints), and no wall crosses both $P^\prime$ and $Q$, by choice of $W$ and the fact that $Q \subset B$, by Lemma~\ref{edge path convex}. 

We see that path $P^\prime Q$ is minimal and does not cross wall $W$ but connects edges $e_1$ and $e_3$, which do cross wall $W$. Thus by Lemma~\ref{boundary N(W)}, we can assert without loss of generality that $P^\prime Q$ lies entirely in $W \times \{1\}$.

Now let $e_2$ be the edge that crosses $W$ and has endpoint $b\in B$ (we know such an edge exists since $b \in P^\prime Q \subset W \times \{1\}$). Recall that wall $W$ intersects $B$, so by Lemma~\ref{wall subcomplex edge}, edge $e_2$ also lies in $B$ (and so $e_3=e_2$).

But then we can modify path $P$ by replacing its final subpath $e_1P^\prime$, which crosses $W$ and then stays in $N(W)\times \{1\}$ until it reaches edge $e_2 \subset B$, with the edge path that does not cross $W$ but instead stays in $N(W)\times \{0\}$ until reaching edge $e_2$. We thus form an edge path from $A$ to $B$ which is strictly shorter than path $P$ (shorter by the width of $W$), so $P$ is not minimal. This is a contradiction; thus $P$ crosses only separating walls.

\end{proof}

Lemma~\ref{min2}, along with our opening discussion, gives us the following proposition.

\begin{proposition}
\label{min}
An edge path $P$ between convex subcomplexes is minimal if and only if it crosses only separating walls, each exactly once. Equivalently, $P$ is minimal if and only if no wall crosses $P$ more than once and no wall intersects both $P$ and one of the subcomplexes.
\end{proposition}

For any two convex subcomplexes $A,B$ of $X$, define $d_1(A, B)$ to be the length of a minimal length edge path joining $A$ and $B$; then Proposition~\ref{min} implies that $d_1(A,B)=\sum_{W\in S_{A,B}}w(W)$, where $S_{A,B}$ is the set of walls that separate $A$ from $B$.

We now expand our notion of edge path ($d_1$) distance so as to define the $d_1$ distance between any two points of $X$, not just between vertices (or subcomplexes). We say that a path is a \emph{type 1 path} if its intersection with each rectangular cell is piecewise linear with each piece parallel to an edge of the cell. Then we let $d_1(x,y)$ be the minimal length of a type 1 path from $x$ to $y$. Note that if $x,y$ are vertices, then this definition of $d_1(x,y)$ agrees with the original (which says that $d_1(x,y)$ is the length of a minimal \emph{edge path} from $x$ to $y$). 

The next lemma will allow us to define $l_1$ length functions in a way that makes them both useful and relatively easy to find.

First, a definition: call a bi-infinite type $1$ path a \emph{type 1 axis for $g$} if it is $g$-invariant and \emph{geodesic}, that is, minimal (among type $1$ paths) between any two points of the path. Note that a type 1 axis for $g$ can cross no wall more than once. 

Also,  if $x\in \m(g)$ and $Q$ is a type 1 geodesic from $x$ to $g.x$, then the bi-infinite path $Q^\infty$ formed by the $g^k$-translates of $Q$ is a type 1 axis for $g$. 

\begin{lemma}
\label{edge path axes}
Let the group $G$ act by isometries on the CAT(0) rectangle complex $X$. Let $g\in G$. Then 
\begin{enumerate}
\item Any two type 1 axes $Q_1^\infty, Q_2^\infty$ for $g$ cross the same set of walls.
\item If $y_j$ is a point of $Q_j^\infty$ ($j=1,2$), then $d_1(y_1, g.y_1)=d_1(y_2, g.y_2)$.
\end{enumerate}
\end{lemma}

\begin{proof}

Let $S=[y_1,y_2]$ be the geodesic from $y_1$ to $y_2$. Let $Q_1\subset Q_1^\infty$ be the path from $y_1$ to $g.y_1$, and let $Q_2 \subset Q_2^\infty$ be the path from $y_2$ to $g.y_2$ (see  Figure~\ref{fig:4}). 
\begin{figure}[htbp]
  \centering
  \includegraphics[width=3in]{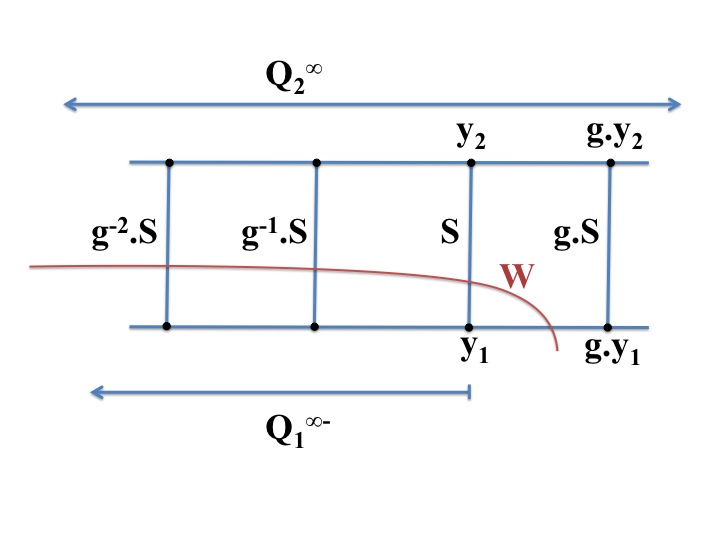}
  \caption{Illustration for Lemma~\ref{edge path axes}.}
  \label{fig:4}
\end{figure}

By symmetry and the fact that both $Q_1^\infty$ and $Q_2^\infty$ are both $g$-invariant, to show (1) it suffices to show that all walls crossed by the type 1 path $Q_1$ are also crossed by $Q_2^\infty$. 

Suppose $Q_1$ crosses wall $W$, and suppose for contradiction that ${Q_2^\infty}$ does not cross wall $W$. Since $Q_1\subset Q_1^\infty$ crosses no wall more than once, $W$ separates its endpoints $y_1$ and $g.y_1$. So the continuous path $S\cdot Q_2  \cdot g.\overline{S}$ from $y_1$ to $g.y_1$ must also cross $W$. Now since $Q_2 \subset Q_2^\infty$ does not cross $W$, either $S$ or $g.S$ must. Without loss of generality, suppose $S$ crosses $W$; since $S$ is geodesic this means that $W$ must separate its endpoints $y_1$ and $y_2$. Note that the ray ${Q_1^\infty}^-$ consisting of the $g^{-i}$-translates of $Q_1$, $i> 0$, does not intersect wall $W$ and contains vertex $y_1$. Also $Q_2^\infty$ does not intersect $W$ and contains $y_2$. Thus $W$ separates $Q_2^\infty$ from ${Q_1^\infty}^-$. Note that for each $i>0$, $g^{-i}.S$ connects ${Q_1^\infty}^-$ to $Q_2^\infty$ and so crosses wall $W$. Thus $g^{-i}.S$ crosses $W$ for all $i>0$; this implies that $S$ crosses $g^i.W$ for all $i>0$. Furthermore the $g^i.W$ are all distinct walls, since for each $i$,  $g^i.Q_1 \subset Q_1^\infty$ crosses $g^i.W$, and since $Q_1^\infty$ crosses no wall more than once. But this is a contradiction, since $S$ has finite length. So $Q_2^\infty$ must cross $W$. This completes the proof of (1).

Continuing with the proof of (2), note that $Q_2^\infty$ is the union of the $g$-translates of ${Q_2}$; let $g^m.{Q_2}$ be the translate that crosses wall $W$. Then path $Q_2$ crosses wall $g^{-m}.W$. 

The above shows that if $W_1, W_2, \dots, W_k$ are the walls crossed by $Q_1$, then walls $g^{-m_1}.W_1$, $g^{-m_2}.W_2$, $\dots$, and $g^{-m_k}.W_k$  are crossed by $Q_2$, for some integers $m_i$. Furthermore these translated walls are all distinct, and the width of a translated wall is just the width of the original. Thus $l(Q_2)\geq \sum_{i=1}^k w(g^{-m_k}.W_k) = \sum_{i=1}^k w(W_i)=l(Q_1)$ (here we are using that the length of a type 1 geodesic from a point to its $g$-translate will be the sum of the widths of the separating walls.)

Since $Q_1$ is minimal between its endpoints $y_1$ and $g.y_1$, its length is $d_1(y_1, g.y_1)$, and similarly $l(Q_2)$ is $d_1(y_2, g.y_2)$; thus we have $d_1(y_2, g.y_2)\geq d_1(y_1, g.y_1)$.

By symmetry, then, we have that $d_1(y_2, g.y_2)=d_1(y_1, g.y_1)$, as desired. 

\end{proof}

The distance $d_1(y, g.y)$ (where $y$ is a point on a type 1 axis for $g$) will be called the \emph{$l_1$-translation distance} along the axis. Note that it does not depend on the choice of the axis or the point $y$, by Lemma~\ref{edge path axes}.

We use translation distances to define $l_1$ length functions. Explicitly, given an action $\alpha$, we define $l^\alpha_1(g)$ to be the $l_1$ translation distance along any type 1 axis for $g$. Note that $l_1^\alpha(g)$ is defined for all $g$, since type 1 axes always exist. In particular every point $x$ of $\m(g)$ is on a type 1 axis for $g$, so that $l_1^\alpha(g)=d_1(x,g.x)$ for any point $x$ of $ \m(g)$.

As observed above, any $g$-invariant, bi-infinite edge path that crosses no wall more than once is an $l_1$-axis, so $l^\alpha_1(g)=d_1(x, g.x)$ for any vertex $x$ of the path.

This fact will allow us to measure the $d_1$ distance between minsets using any $g$ and $h$ for which the associated natural $gh$-invariant bi-infinite edge path crosses no wall more than once. 

Note that if $J\subset \G$ is a join of $n$ discrete subgraphs $\G_1, \G_2, \dots , V_n$ with vertex sets $V_1, V_2, \dots , V_n$, 
then $A_J$ is a product of free groups: $A_J=F(V_1)\times F(V_2)\times \dots \times F(V_n)$. In particular, any $n$-tuple of elements $h_1 \in F(V_1), h_2 \in F(V_2), \dots, h_n \in F(V_n)$ generates a $\Z^n$-subgroup. The converse is not quite true: that is, it is not true that every $\Z^n$-subgroup of $\AG$ is contained in some $A_J$, where $J$ is a join of $n$ discrete graphs.  However, something very close is true: every $Z^n$-subgroup is \emph{conjugate} to a subgroup of some $A_J$, as we'll see in Lemma~\ref{4.2}.

\begin{definition}
We say that a subgroup $H < \AG$ is a \emph{basic $\Z^n$-subgroup} if $H$ is a maximal  $\Z^n$-subgroup and is contained in $A_J \cong F(V_1)\times F(V_2)\times \dots \times F(V_n)$ for some join $J \subset \Gamma$ of $n$ non-empty discrete graphs $\G_1, \G_2, \dots , \G_n$ with vertex sets $V_1, V_2, \dots , V_n$. 
\end{definition}

\begin{lemma}
\label{4.2}
Let $\G$ be a simplicial graph and let $\AG$ be the corresponding right-angled Artin group. Then 
\begin{enumerate} 
\item Any abelian subgroup $H< A_\G$ is conjugate to a subgroup of $\langle a_1 \rangle \times \dots \times \langle a_m \rangle <  A_{\G_1} \times A_{\G_2}\times \dots \times A_{\G_m}=A_J$, where $J\subset \Gamma$ is the join of  graphs $\G_1, \dots, \G_m$ (for some $m$).  
\item If $\G$ is homogeneous of dimension $n$ and $H$ is a maximal $\Z^n$-subgroup, then $H$ is conjugate to $\langle a_1 \rangle \times \dots \times \langle a_n \rangle <  A_J$, where $J\subset \Gamma$ is the join of non-empty discrete graphs $\G_1, \dots, \G_n$; in particular, $H$ is conjugate to a basic $\Z^n$-subgroup.
\end{enumerate}
\end{lemma}

\begin{proof}

We first show (1) by induction on $l$, the number of vertices of $\G$. When $l=1$ there is nothing to show.

Let  $\G$ have $l>1$ vertices, and let $H< A_\G$ be abelian. Clearly $H$ is conjugate to some subgroup of $\AG$ that contains a cyclically reduced element; we now assume that $H$ itself contains the cyclically reduced element $h$.  $H$ is then contained in the centralizer $C(h)$ of $h$. We now use a theorem of Servatius \cite{Serv} that describes the centralizer $C(h)$ of a cyclically reduced element $h$ as follows. Let $supp(h)$ be the set of vertices of $\Gamma$ that appear in a minimal length word for $h$, and let $\G_{supp(h)} \subset \G$ be the subgraph spanned by $supp(h)$. Let $k$ be maximal such that $\G_{supp(h)}$ is a join of $k$ subgraphs $\G_1$, $\G_2$, $\dots$, and $\G_k$; then $A_{\G_{supp(h)}}\cong A_{\G_1} \times A_{\G_2}\times \dots \times A_{\G_k}$. Express $h$ as ${a_1}^{m_1} {a_2}^{m_2}\dots {a_k}^{m_k}$, where $a_i \in A_{\G_i}$ and $m_i$ is maximal for each $i$. Also, let $lk(h)$ be the set of vertices of $\G$ that commute with each vertex of $supp(h)$. Then $C(h)=\langle a_1, a_2, \dots, a_k, lk(h) \rangle$. So we have that $H < C(h) \cong \langle a_1 \rangle \times \dots \times \langle a_k \rangle \times A_{\G_{lk(h)}} \subset A_{\G_1} \times \dots \times A_{\G_k}\times A_{\G_{lk(h)}}=A_{J^{\prime \prime}}$, where $J^{\prime \prime}\subset \Gamma$ is the join of ${\G_1}, \dots, {\G_k}, {\G_{lk(h)}}$. 

Let $p$ be the projection onto the last factor $A_{\G_{lk(h)}}$. Then $H <  \langle a_1 \rangle \times \dots \times \langle a_k \rangle \times p(H)$. Note that $p(H)$ is an abelian subgroup of $A_{\G_{lk(h)}}$, where $\G_{lk(h)}$ has fewer than $l$ vertices. So by the inductive hypothesis, $p(H)$ is conjugate by an element of $A_{\G_{lk(h)}}$ to a subgroup of $\langle a_{k+1} \rangle \times \langle a_{k+2} \rangle \times \dots \times \langle a_{m} \rangle  <  A_{\G_{k+1}} \times \dots \times A_{\G_m}=A_{J^\prime}$, where $J^\prime \subset \G_{lk(H)}$ is the join of  graphs $\G_{k+1}, \dots, \G_m$ (for some $m$). Thus $H$ is conjugate to a subgroup of $\langle a_1 \rangle \times \dots \times \langle a_m \rangle <   A_{\G_{1}} \times \dots \times A_{\G_m}=A_J$, where $J$ is the join of $\Gamma_1, \dots, \Gamma_m$. This proves (1).

Now assume that $\G$ is homogeneous of dimension $n$ and $H$ is a maximal $\Z^n$-subgroup. 

Given these assumptions, we want to show that $m$ must equal $n$. Clearly  $m \geq n$, since $H\cong \Z^n$ is conjugate to a subgroup $H^\prime$ of $\langle a_1 \rangle \times \dots \times \langle a_m \rangle \cong \Z^{m}$. Also $m \leq n$, since $\G$ is homogeneous of dimension $n$ and contains the join of $\Gamma_1, \dots, \Gamma_m$. Thus $m = n$, and we have that $ H^\prime < \langle a_1 \rangle \times \dots \times \langle a_n \rangle< A_J $, where $J$ is the join of  $\Gamma_1, \dots, \Gamma_n$. In fact, since $H^\prime \cong \Z^n$ is maximal, $H^\prime$ is exactly $\langle a_1 \rangle \times \dots \times \langle a_n \rangle$. 

It remains to show that each $\G_i$ is discrete.

Suppose for contradiction that some $\G_j$ is not discrete, so that it contains at least one edge. Let $u, v$ be the endpoints of this edge. Then $\G$ contains the join of the $n+1$ graphs $\G_i$ ($i\neq j$), $\{u\}$, and $\{v\}$, a contradiction since $\G$ is homogeneous of dimension $n$. Thus each $\G_i$ must be discrete. 
 
This completes the proof of (2).
\end{proof}

We will say that an edge path \emph{traverses} an $n$-dimensional rectangle $R$ if the edge path passes through diagonally opposite vertices of the rectangle.

\begin{lemma}
\label{edge path crossing n-rectangle}
In an $n$-dimensional rectangle complex, a minimal edge path that begins and ends in the same component of $N(W)$ cannot traverse an $n$-dimensional rectangle. 
\end{lemma}

\begin{proof}
By Corollary~\ref{boundary N(W)}, the minimal edge path $P$ lies entirely in a boundary component of $N(W)$, say $W \times \{1\}$. Since $W \times \{1\}$ is a convex subcomplex of dimension $n-1$, it cannot contain an $n$-dimensional rectangle.
\end{proof}

\begin{lemma}
\label{m(g) subcomplex}
Let $\G$ be homogeneous of dimension 2. Let $g\in \AG$ such that the centralizer of $g$ is non-cyclic. Then $\m(g)$ is a convex subcomplex of $X$.
\end{lemma}

\begin{proof}

$\m(g)$ is convex by Bridson and Haefliger \cite{BH}. It remains to show that it is a subcomplex. 

By a theorem in Bridson and Haefliger \cite{BH}, $\m (g)\cong T \times U$, where $T$ is a tree and $U \cong \R$. Since the centralizer of $g$ is non-cyclic, $g$ is an element of a $\Z^2$-subgroup $A$ of $\AG$. By the Flat Torus Theorem, $\m(A)\cong \E^2$, so $\E^2 \subset \m(g)$, so $T$ is not a single point.

Assume the tree $T$ is branched, and $t^*\in T$ is a branch point. Then $\{t^*\}\times U$ must be an edge path. Then by the rectangle complex structure of $X$, if $(t^*,u^*)$ is any vertex of $\m(g)$, $T\times \{u^*\}$ also lies in the 1-skeleton of $X$. The only way, then, for $\m(g)$ to fail to be a subcomplex is if $T$ has a leaf $t^{\prime}$ such that $(t^\prime, u^*)$ is not a vertex of $X$. But this would imply that some rectangular cell of $X$ fails to lie in $\m(g)$ but contains some nondegenerate strip contained in $\m(g)$, impossible since $g$ acts simplicially on $X$. Therefore $\m(g)$ is a subcomplex of $X$.

If $T$ is not branched, then $\m(g) \subset \R \times \R \cong \E^2$, but then $\m(g)\cong \E^2$, so $\m(g)$ must be a subcomplex of $X$.

\end{proof}

\begin{corollary}
\label{m(v) subcomplex}

Let $\Gamma$ be homogenous of dimension 2. Then for any vertex $v$ of $\Gamma$, $\m(v)$ is a convex
subcomplex of $X$.

\end{corollary}

Fix a vertex $v$ in $\G$. We define some useful notation. If $\AG$ acts on $X$ and $u$ is a vertex of $\Gamma$, then define $P^u$ as follows. If $\m(u) \cap \m(v)=\emptyset$, then let $P^{u}=\{x\in \m(v) : x $ is an endpoint in $\m(v)$ of a spanning geodesic from $\m(v)$ to $\m(u)\}$; recall that a \emph{spanning geodesic} from $A$ to $B$ is a minimal length CAT(0) geodesic that joins a point of $A$ to a point of $B$. Alternatively, if $\m(u) \cap \m(v) \neq \emptyset$, then let $P^{u}=\m(u) \cap \m(v)$. 

\begin{lemma}
\label{P^u}
Let $\G$ be homogeneous of dimension 2, and consider vertices $u, v$ of $\G$. Then $P^u$ is a convex subcomplex of $\m(v)\cong T \times U$; thus $P^u\cong T^u \times U^u$, where $T^u$ and $U^u$ are subtrees of $T$ and $U$, respectively. If $\m(u) \cap \m(v)=\emptyset$ then $P^u$ is contained in the 1-skeleton of $\m(v)$, so either $T^u$ or $U^u$ consists of a single vertex.

\end{lemma}

\begin{proof}

First suppose that $\m(u) \cap \m(v) \neq \emptyset$. Then $P^{u}=\m(u) \cap \m(v)$ is a convex subcomplex of $X$, since both $\m(u)$ and $\m(v)$ are (by Corollary~\ref{m(v) subcomplex}). 

Now suppose that $\m(u) \cap \m(v) = \emptyset$. 

First note that in this case, any point of $P^u$ must lie in the 1-skeleton of $\m(v)$, since a point in the interior of a 2-dimensional rectangular cell cannot be an endpoint of a spanning geodesic.

Recall that the union of the spanning geodesics between two convex subcomplexes is called the bridge between them. By the Flat Quadrilateral Theorem, the bridge between $\m(u)$ and $\m(v)$ is convex; thus $P^u$ which is the intersection of the bridge with $\m(v)$ is convex.

To show that $P^u$ is a subcomplex of $X$, we need only show that if $P^u$ contains a point in the interior of an edge, then it contains the entire edge. Suppose $P^u$ contains a point $x$ in the interior of edge $[a,b]$, so that $x$ is the endpoint of a spanning geodesic. $P^u$ also contains some vertex, since there must exist a spanning geodesic with vertex endpoints, by \cite{ChM}. Because $P^u$ is a connected subset of the 1-skeleton that contains both $x$ and some vertex, then $P^u$ must contain one of the endpoints of the edge $[a,b]$ containing $x$. Without loss of generality, suppose $a \in P^u$. Then the spanning geodesics starting from $a$ and $x$ form opposite sides of a rectangle in $X$. Since the third side $[a,x]$ lies in the 1-skeleton, by the rectangle structure of $X$, the spanning geodesic with endpoint $a$ must be an edge path. 

Then the spanning geodesic from $a$ must bound a strip of rectangles whose union contains the spanning geodesic from $x$ in its interior. Then the opposite side of the strip, which is an edge path starting from vertex $b$, is also a spanning geodesic (since it has the same length as the spanning geodesic from $a$ and also joins $\m(v)$ and $\m(u)$). Thus $b \in P^u$. 

Since $P^u$ is convex and $a\in P^u$, the entire edge $[a,b]$ lies in $P^u$, which is what we wanted to show. Thus $P^u$ is a subcomplex of $X$.

Using either definition of $P^u$, since $P^u \subset \m(v)\cong T \times U$, the convexity of $P^u$ implies that it must be a product $T^u \times U^u$, where $T^u$ and $U^u$ are subtrees of $T$ and $U$, respectively.

Note that in the case that $\m(u) \cap \m(v) = \emptyset$, since $P^u$ lies in the 1-skeleton of $\m(v)$, either $T^u$ or $U^u$ consists of a single vertex.

\end{proof}

Recall that $P^u$ is the set of endpoints in $\m(v)$ of spanning geodesics from $\m(v)$ to $\m(u)$.

\begin{lemma}
\label{P^u vertices}
 Let $\Gamma$ be homogeneous of dimension 2, and let $u,v$ be vertices of $\Gamma$. The set of endpoints in $\m(v)$ of minimal edge paths from $\m(v)$ to $\m(u)$ is equal to the set of vertices of $P^u$.
\end{lemma}

\begin{proof}
($\subset$):
First note that (i) if $g_1$ and $g_2$ are geodesic segments joining pairs of vertices, and if the set of walls that cross $g_2$ contains the set that cross $g_1$, then $l(g_1)\leq l(g_2)$. 

Also note that (ii) since there exists a spanning geodesic \emph{with vertex endpoints} (by \cite{ChM}), if a geodesic $g$ has minimal length among geodesics joining \emph{vertices} of $\m(u)$ and $\m(v)$, then $g$ has minimal length among \emph{all} geodesics joining $\m(u)$ and $\m(v)$, that is, $g$ is a spanning geodesic.

Recall that a minimal edge path between $\m(v)$ and $\m(u)$ crosses only the walls that separate $\m(u)$ from $\m(v)$; thus so does the geodesic segment $g$ with the same endpoints. Since any geodesic joining $\m(v)$ to $\m(u)$ must cross at least these walls, $g$ must have minimal length among geodesics between vertices of $\m(v)$ and $\m(u)$, by (i). Thus by (ii) $g$ is a spanning geodesic.

This shows that an endpoint of a minimal edge path is also an endpoint of a spanning geodesic.

($\supset$):
We look at two cases. 

First suppose that there is a unique spanning geodesic. Since any minimal edge path gives a spanning geodesic with the same endpoints (by the above paragraph), all minimal edge paths must join the endpoints of the unique spanning geodesic. 

Now suppose that there is more than one spanning geodesic. By the Flat Quadrilateral Theorem (Bridson and Haefliger) and Lemma~\ref{P^u}, the union of the spanning geodesics is a product $P^u \times S$, where $S$ is a spanning geodesic and $P^u$ is a tree in the 1-skeleton of $X$. Thus by the rectangular cell structure of $X$, any spanning geodesic that starts at a vertex of $P^u$ must actually be an edge path, and so must be a minimal edge path between $\m(v)$ and $\m(u)$.

This shows that every vertex of $P^u$ is an endpoint of a minimal edge path from $\m(v)$ to $\m(u)$.

\end{proof}

We define some important notation. Given $\AG$ a 2-dimensional RAAG, let $D$  be the set of all elements of $\AG $ of reduced word length less than or equal to 4, 
and given an action $\alpha: \AG \circlearrowleft X$, define 

$$ M_1^\alpha =\max_{d\in D}l_1^\alpha(d).$$

\begin{theorem}
\label{a1}
Let $\G$ be a simplicial graph that is homogeneous of dimension 2. Let $\AG$ be the RAAG with defining graph $\G$. Let $\alpha$ be a minimal action by $\AG$ on a 2-dimensional CAT(0) rectangle complex $X$.

Then for any two vertices $u$, $w$ of $\G$, $d_1(\m(u), \m(w)) \leq M_1^\alpha$. 
\end{theorem}

\begin{proof}
We may assume that $\m(u) \cap \m(w) = \emptyset$, because otherwise  $$d_1(\m(u), \m(w))=0$$ and the inequality holds. 

Take any vertex $u^\prime \in lk(u)$, and note that $\E^2 \cong \m \langle u,u^\prime \rangle \subset \m(u)$.

Claim 1: some element $u^*$ of reduced word length 1 or 2 has $\m(u^*) = \m \langle u, u^\prime \rangle$. Proof: if $u$ is \emph{not} a gridline isometry, then take $u^*=u$; if $u^\prime$ is \emph{not} a gridline isometry, then take $u^*=u^\prime$. If both $u, u^\prime$ are gridline isometries, then take $u^*=u u^\prime$. Note that however $u^*$ is defined, $\E^2 \cong \m \langle u, u^\prime \rangle \subset \m(u^*)$. Equality follows since $\m(u^*)\cong \E^2$. This is because $\m(u^*)\cong \E \times U$, where $U$ is a tree and $u^*$ acts in the direction of the factor $\E$. But if $U \ncong \E$, then there is branching along some line $\E \times \{u\}$ so this line must be a gridline and $u^*$ is a gridline isometry, contradicting our definition of $u^*$; thus $U\cong \E$ and $\m(u^*)\cong \E^2$.

Similarly, there is some element $w^* $  of reduced word length 1 or 2 such that $\m(w^*)=\m \langle w, w^\prime \rangle \cong \E^2$, where $w^\prime \in lk(w)$.

We will show (Claim 2) that either $u^*w^*$ or $u^*(w^*)^{-1}$ has $l^\alpha_1$ translation length greater than or equal to $d_1(\m(u^*), \m(w^*))$; this in turn is clearly greater than or equal to $d_1(\m(u), \m(w))$ since $\m(u^*)\subset \m(u)$ and $\m(w^*)\subset \m(w)$. The theorem follows, since $u^*w^*, u^*(w^*)^{-1} \in D$.

Now we prove Claim 2.

Let $P$ be any minimal edge path from $\m(u^*)$ to $\m(w^*)$, so that 
$$d_1(\m(u^*), \m(w^*)) =l(P).$$  Let $x\in \m(u^*)$ and $y \in \m(w^*)$ be the endpoints of $P$.

We now construct an edge path $Q$ from $x$ to $u^*w^*.x$. Let $A$ be a minimal edge path from $x$ to $u^*.x$, and let $B$ be a minimal edge path from $y$ to $w^*.y$ (see  Figure~\ref{fig:5}). 
\begin{figure}[htbp]
  \centering
  \includegraphics[width=2.5in]{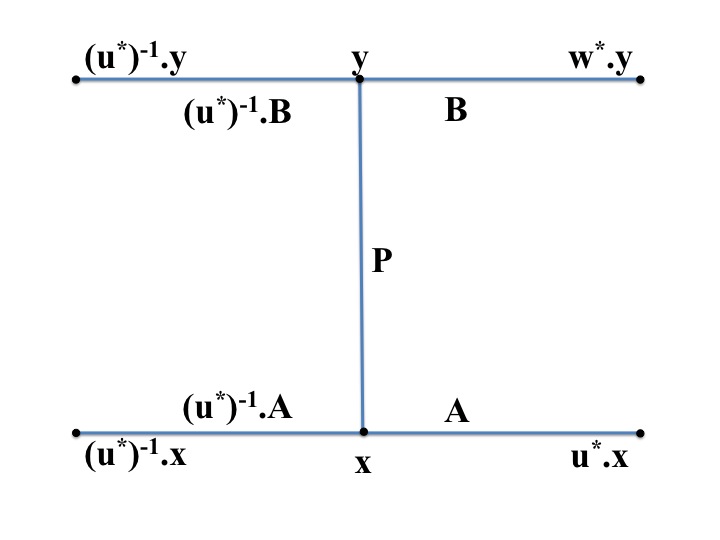}
  \caption{Illustration for Lemma~\ref{a1}, edge paths $A$, $B$, and $P$.}
  \label{fig:5}
\end{figure}
Let $Q$ be the edge path $A \cdot u^*.P \cdot u^*.B \cdot u^*w^*.\overline{P}$, which joins $x$ to $u^*w^*.x$ (see  Figure~\ref{fig:6}). 
\begin{figure}[htbp]
  \centering
  \includegraphics[width=3in]{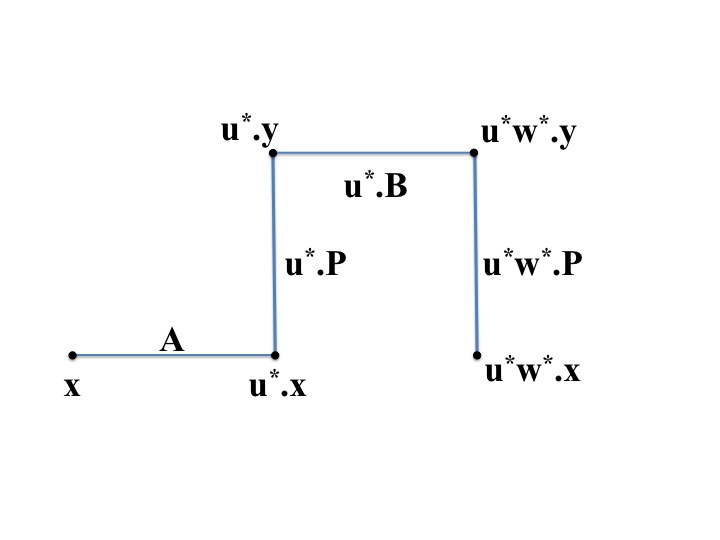}
  \caption{Illustration for Lemma~\ref{a1}, edge path $Q$.}
  \label{fig:6}
\end{figure}
 Here $\overline{P}$ denotes the path with the same image as path $P$ but that runs in the opposite direction. 

Since $Q$ contains translates of $P$ as subpaths, $l(P)\leq l(Q)$.  We concatenate the $u^*w^*$-translates of $Q$ to form the infinite edge path $Q_\infty$.

If no wall crosses $Q_\infty$ more than once, then (1) $Q\subset Q_\infty$ is minimal so $l(Q)=d_1(x, u^*w^*.x)$, and (2) by Lemma~\ref{edge path axes}, $d_1(x, u^*w^*.x)= l^\alpha_1(u^*w^*)$; thus we have $d_1(\m(u^*), \m(w^*))=l(P) \leq l(Q)=d_1(x, u^*w^*.x)= l^\alpha_1(u^*w^*)$.

Now suppose that some wall crosses $Q_\infty$ more than once. 
We will show that we can construct an alternative edge path $Q^\prime$ that connects $x$ to $u^*(w^*)^{-1}.x$, that contains a translate of $P$ as a subpath, and for which no wall crosses $Q^\prime_\infty$ more than once. Then by the same reasoning as for $Q$, we can conclude that $d_1(\m(u^*), \m(w^*)) \leq l^\alpha_1(u^*(w^*)^{-1})$. This will complete the proof of Claim 2.

Let $S$ be a minimal segment of $Q_\infty$ that crosses the same wall twice, and let $W$ be the wall that it crosses. Note that $S$ consists of an $l_1$-geodesic subpath $S^\prime$ of $Q_\infty$ flanked on both sides by edges that cross wall $W$. 

Since edge paths $A $ and $B$ traverse ($2$-dimensional) rectangles, Lemma~\ref{edge path crossing n-rectangle} implies that $S$ cannot contain either $A$ or $B$ or any of their translates. Also, (1) wall $W$ cannot cross $A$, $B$, or $P$ (or their translates) more than once, since by definition, these edge paths are minimal, and (2) wall $W$ cannot cross both $P$ and the adjoining translate of either $A$ or $B$, by Lemma~\ref{min2}, since $P$ is a minimal edge path from $\m(u^*)$ to $\m(w^*)$. Thus wall $W$ must cross both some translate of $A$ and one of the neighboring translates of $B$. We may assume without loss of generality that $S$ contains $P$ rather than some translate of $P$. Then there are two possibilities: either $W$ crosses both $(u^*)^{-1}.A$ and $B$, or $W$ crosses both $(w^*)^{-1}.B$ and $A$.

For simplicity we treat only the case where wall $W$ intersects both $(u^*)^{-1}.A$ and $B$; the other case is treated in exactly the same way.

Lemma~\ref{boundary N(W)} says that $S$ lies in $N(W)$, and Lemma~\ref{wall subcomplex edge} says that then $W$ crosses both some edge of $\m(u^*)$ that has endpoint $x$ and some edge of $\m(w^*)$ that has endpoint $y$ (see Figure~\ref{fig:7}; the dotted line is wall $W$). 
\begin{figure}[htbp]
  \centering
  \includegraphics[width=5in]{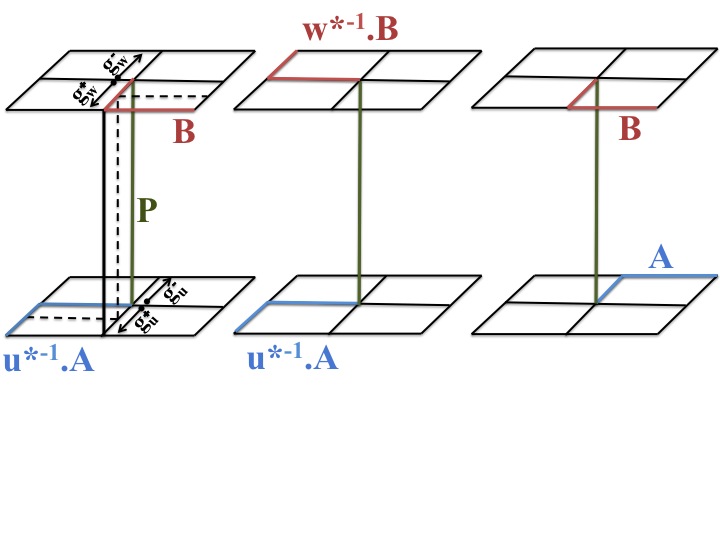}
  \caption{Illustration for Lemma~\ref{a1}, wall $W$.}
  \label{fig:7}
\end{figure} 
Since $P$ is a minimal edge path between $\m(u^*)$ and $\m(w^*)$ that connects $x$ to $y$, and since $N(W)\cong W \times [0,1]$, the path `opposite' $P$ in $N(W)$ is also a minimal edge path between $\m(u^*)$ and $\m(w^*)$.

By Lemmas~\ref{P^u} and \ref{P^u vertices}, then there is a (nondegenerate) Euclidean strip of spanning geodesics that connect $\m(u^*)$ and $\m(w^*)$ (of which one spanning geodesic lies in $W$). Furthermore, by Lemma~\ref{P^u}, the set of endpoints of these spanning geodesics in $\m(u^*)$ or $\m(w^*)$  lies in a single gridline of $\m(u^*)$ or $\m(w^*)$; call the gridline $g_u$ or $g_w$, respectively. Note that the strip of spanning geodesics allows us to assign consistent directions to $g_u$ and $g_w$; we use $x$ (respectively, $y$) as the origin of the line $g_u$ ($g_w$, respectively) and assign orientation by designating by $g^+_u$  (respectively, $g^+_w$) the half-line of $g_u$ ($g_w$, respectively) that is intersected by $W$. Recall that wall $W$ was assumed to intersect both $(u^*)^{-1}.A$ and $B$, so that $(u^*)^{-1}.A$ and $g^+_u$ share a wall, and $B$ and $g^+_w$ share a wall. Thus, taking note of the rectangular grid structure of $\m(u^*)\cong \E^2$, we see that \emph{any} wall that crosses both  $(u^*)^{-1}.A$ and $g_u$ must cross $g^+_u$, and \emph{any} wall that crosses both  $A$ and $g_u$ must cross $g^-_u$. And similarly, \emph{any} wall that crosses both  $(w^*)^{-1}.B$ and $g_w$ must cross $g^-_w$, and \emph{any} wall that crosses both  $B$ and $g_w$ must cross $g^+_w$.

Now define the edge path $Q^\prime= A \cdot u^*.P \cdot u^*(w^*)^{-1}.\overline{B} \cdot u^*(w^*)^{-1}.\overline{P}$; this path connects $x$ to $u^*(w^*)^{-1}.x$. We claim that no wall crosses the path $Q^\prime_\infty$ more than once.

Suppose for contradiction that indeed some wall crosses the path $Q^\prime_\infty$ more than once. Then by the same reasoning as above, some wall $W^\prime$ must contain a spanning geodesic that connects $\m(u^*)$ and $\m(w^*)$, so $W^\prime$ crosses both $g_u$ and $g_w$.

Also, $W^\prime$ must cross either both $(u^*)^{-1}.A$ and $(w^*)^{-1}.B$, or both $B$ and $A$ (see Figure~\ref{fig:7}). In the first case, $W^\prime$ crosses both $g^+_u$ and $g^-_w$, and in the second case, $W^\prime$ crosses both $g^-_u$ and $g^+_w$. In either case, we have a contradiction, since any wall that crosses $g^+_u$ must also cross $g^+_w$ (and not $g^-_w$), and any wall that crosses $g^-_u$ must also cross $g^-_w$ (and not $g^+_w$), since the space bounded by $g_u$ and $g_w$ has the structure of a Euclidean strip.

Thus, no wall crosses $Q^\prime_\infty$ more than once. Since $Q^\prime$ connects $x$ to $u^*(w^*)^{-1}.x$ and contains a translate of $P$ as a subpath, as noted before we can conclude that $d_1(\m(u^*), \m(w^*)) \leq l^\alpha_1(u^*(w^*)^{-1})$, and our proof is complete.

\end{proof}

We need the following lemma in order to choose a basepoint $x_0 \in X$ for a given action on $X$.

\begin{lemma}
\label{trees}
If $\{T_i\}_{i=1}^n$ is a collection of subtrees of a tree $T$, and $\cap_{i=1}^nT_i = \emptyset$, then for some $i$ and $j$, $T_i \cap T_j = \emptyset$.
\end{lemma}

\begin{proof}
Since $\cap_{i=1}^nT_i = \emptyset$, there exists $1 \leq k \leq n-1$ such that $\cap_{i=1}^{k}T_i \neq \emptyset$ but $(\cap_{i=1}^{k}T_i)\cap T_{k+1} = \emptyset$. Consider the spanning geodesic between $\cap_{i=1}^{k}T_i$ and $T_{k+1}$. Clearly if $p$ is a point in the interior of the spanning geodesic, then $\cap_{i=1}^{k}T_i$ and $T_{k+1}$ are in different connected components of $T \setminus \{p\}$ (since both $\cap_{i=1}^{k}T_i$ and $T_k$ are connected). But since $p \notin \cap_{i=1}^{k}T_i$, there exists $1 \leq j \leq k$ such that $p \notin T_j$. So $T_j$ is entirely contained in the connected component of $T \setminus \{p\}$ that contains $\cap_{i=1}^{k}T_i$, so $T_j$ and $T_{k+1}$ lie in different connected components of $T \setminus \{p\}$, so $T_j \cap T_{k+1}= \emptyset$.
\end{proof}

Now choose $x_0\in \m (v) \cong T\times U$ by choosing both its tree-coordinate and its $U$-coordinate appropriately, as
follows. If $u \neq v$ is a vertex of $\Gamma$, then according to Lemmas~\ref{P^u} and~\ref{P^u vertices}, $P^u=T^u \times U^u$, where $T^u$ and $U^u$ are subtrees of $T$ and $U$, respectively. If $\cap_{u\in V \setminus \{v\}}T^u \neq \emptyset$, then let the tree-coordinate $x_{0,1}$ of $x_0$ be any vertex of $\cap_{u\in V \setminus \{v\}}T^u$. If $\cap_{u\in V \setminus \{v\}}T^u = \emptyset$, then by Lemma~\ref{trees}, there exist $t, t^\prime$ such that $T^t\cap T^{t^\prime} = \emptyset$. Choose such a pair $(t,t^\prime)$. Then let $x_{0,1}$ be any vertex on the spanning geodesic from $T^t$ to $T^{t^\prime}$. Define the second coordinate $x_{0,2}$ of $x_0$ through a similar process.

What follows is a useful technical lemma. Recall that $v$ is a fixed vertex of $\G$. 

\begin{lemma} 
\label{AB}
Let $u$ be any vertex of $\G$ besides $v$. Let $p$ be any vertex of $\m(v)$ (see Figure~\ref{fig:8}). 
\begin{figure}[htbp]
  \centering
  \includegraphics[width=5in]{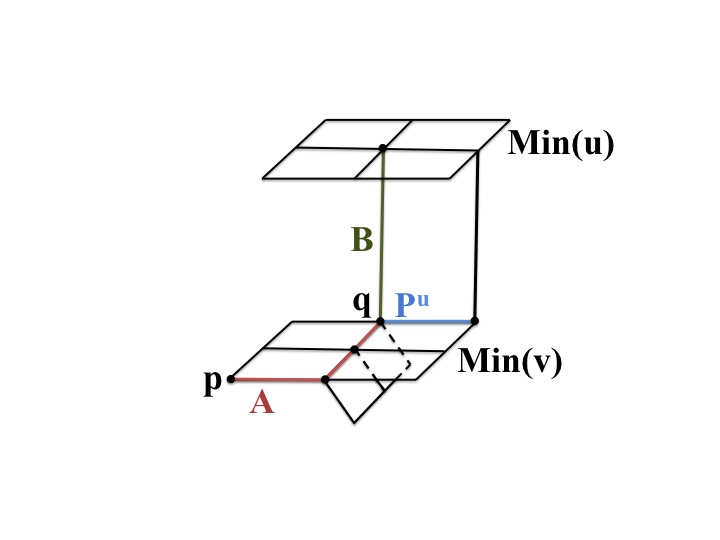}
  \caption{Illustration for Lemma~\ref{AB}.}
  \label{fig:8}
\end{figure}  
Let $A$ be a minimal edge path from $p$ to $P^u$, ending at $q$. Let $B$ be a minimal edge path between $\m(v)$ and $\m(u)$ with endpoint $q$  (we can choose such a $B$ since the vertices of $P^u$ are endpoints of such minimal edge paths, by Lemma~\ref{P^u vertices}). Then $AB$ is a minimal edge path from $p$ to $\m(u)$.
\end{lemma}

\begin{proof}
By Proposition~\ref{min}, we need only show that no wall crosses $AB$ more than once, and no wall intersects both $AB$ and $\m(u)$.

Since we chose $A$  to be minimal between $p$ and $P^u$,  and $B$ to be minimal between  $\m(v)$ and $\m(u)$, by Proposition~\ref{min} no wall crosses either $A$ or $B$ more than once, and no wall crosses both $B$ and $\m(v)$. Note that $A \subset \m(v)$, thus no wall crosses both $B$ and $A$. So no wall  crosses path $AB$ more than once.

Proposition~\ref{min} ensures that no wall crosses both $B$ and $\m(u)$. Finally, 
suppose for contradiction that some wall intersects both $A$ and $\m(u)$. Since $A \subset \m(v)$, then this wall intersects both $\m(v)$ and $\m(u)$, so it must contain a spanning geodesic between them, 
and so must intersect $P^u$. But this is a contradiction, since no wall intersects both $A$ and $P^u$, by Proposition~\ref{min}, since $A$ was chosen to be minimal between $p$ and $P^u$. Thus no wall intersects both  $AB$ and $\m(u)$.

\end{proof}

Recall that $D$ is the set of all elements of $\AG$ of reduced word length less than or equal to 4, and that for an action $\alpha$, $M_1^\alpha = \max_{d\in D}l_1^\alpha(d)$.

\begin{lemma}
\label{BC}
For any vertex $u$ of $\G$, $d_1(x_0, \m (u))\leq 3 M_1^\alpha$.
\end{lemma}

\begin{proof}

$x_0$ plays the role of $p$  in Lemma~\ref{AB}; we form a minimal edge path joining $x_0$ to $\m(u)$ by concatenating any minimal edge path $Q_1$ from $x_0$ to $P^u$ with a minimal edge path $Q_2$ from $\m(v)$ to $\m(u)$. 

By Lemma~\ref{AB}, $d_1(x_0, \m(u)) =l(Q_1)+l(Q_2)$.

By Theorem~\ref{a1}, $l(Q_2) =d_1 (\m(v), \m(u))\leq M_1^\alpha$. 

And since $P^u$ is a product $T^u \times U^u$, $l(Q_1)=d^T_1(x_{0,1},T^u)+d^U_1(x_{0,2},U^u)$, where $d^T_1$ is the edge path distance in $T$, and $d^U_1$ is the edge path distance in $U$.

Let $(t, t^\prime)$ be the pair chosen in the definition of $x_{0,1}$. Either $x_{0,1} \in T^u$, in which case $d^T_1(x_{0,1},T^u)=0$, or $T^u$ is a subset of a single component of $T \setminus \{x_{0,1}\}$. In this latter case, 
either  $T^t$ and $T^u$ lie in different components of $T \setminus \{x_{0,1}\}$, or $T^{t^\prime}$ and $T^u$ do (see Figure~\ref{fig:9}). 
\begin{figure}[htbp]
  \centering
  \includegraphics[width=3in]{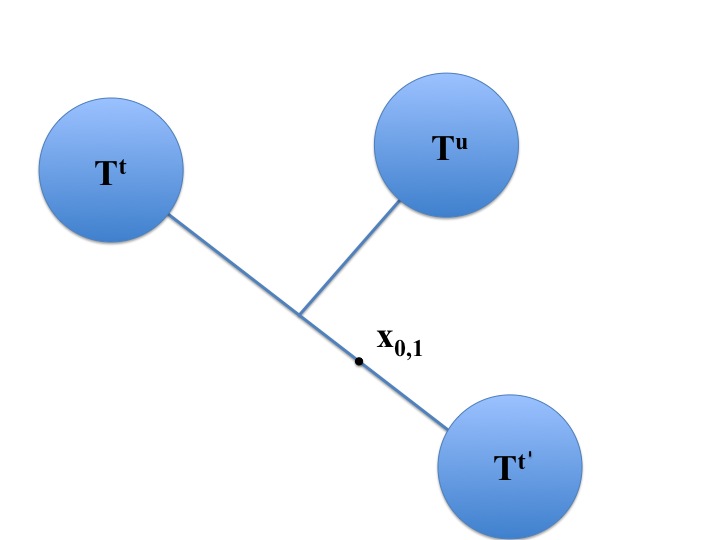}
  \caption{Illustration for Lemma~\ref{BC}, the tree $T$.}
  \label{fig:9}
\end{figure}  Suppose without loss of generality that  $T^t$ and $T^u$ lie in different components of $T \setminus \{x_{0,1}\}$. Then the minimal edge path from $T^u$ to $T^t$ passes through $x_{0,1}$, so $d_1(x_{0,1},T^u)\leq d_1(T^t,T^u)$.

Also, since $P^u$ and $P^t$ are products,  $d_1(T^t,T^u)\leq d_1(P^t, P^u)$.

Now we wish to show that $d_1(P^t,P^u) \leq d_1(\m(t), \m(u))$.
Let $G$ be a minimal edge path from  $P^u$ to $P^t$, with endpoints  $a \in P^u$ and $b\in P^t$ (see Figure~\ref{fig:10}). 
\begin{figure}[htbp]
  \centering
  \includegraphics[width=5in]{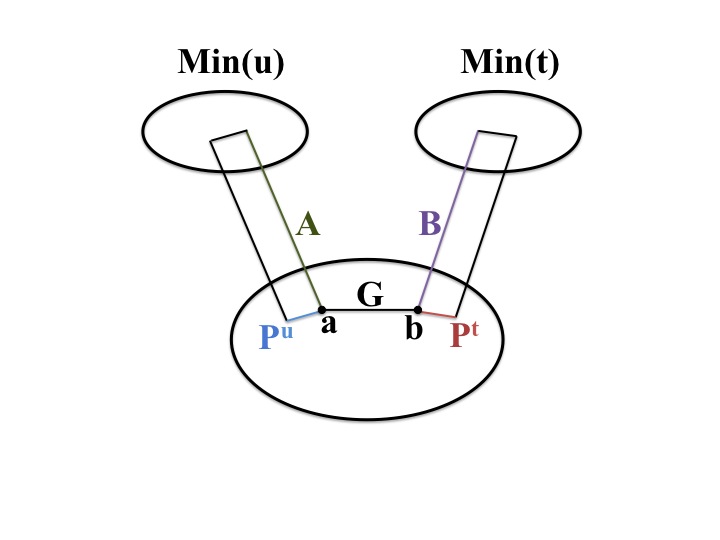}
  \caption{Illustration for Lemma~\ref{BC}, minsets $\m(v)$, $\m(u)$, and $\m(t)$.}
  \label{fig:10}
\end{figure}  
 Let $A$ be a minimal edge path from $\m(u)$ to $\m(v)$ with endpoint $a$, and let $B$ be a minimal edge path from $\m(v)$ to $\m(t)$ with starting point $b$. By Lemma~\ref{AB}, $G$ cannot cross any wall that either crosses $A$  or intersects $\m(u)$. Similarly, $G$ cannot cross any wall that either crosses $B$ or intersects $\m(t)$. Thus every wall that crosses $G$ separates $\m(u)$ from $\m(t)$.  This shows that $d_1(P^t,P^u) \leq d_1(\m(t), \m(u))$.

Thus we have that $d_1(x_{0,1},T^u)\leq d_1(T^t,T^u)  \leq d_1(P^t, P^u)\leq d_1(\m(t), \m(u))$. 

A similar argument shows that there exists a vertex $s$ of $\Gamma$ such that $d_1(x_{0,2},U^u)\leq d_1(\m(u), \m(s))$. 

So $l(Q_1)=d^T_1(x_{0,1},T^u)+d^U_1(x_{0,2},U^u)  \leq d_1(\m(u), \m(t))+d_1(\m(u), \m(s))$, for some vertices $s,t$ of $\Gamma$.

So we have 
\begin{align*}
d_1(x_0, \m(u))& =l(Q_1)+l(Q_2)\\
& \leq d_1(\m(u), \m(t))+d_1(\m(u), \m(s))+d_1 (\m(v), \m(u))\\
& \leq 3 M_1^\alpha.
\end{align*}
 
\end{proof}

Recall that for $g\in \AG$, $|g|$ is its reduced word length and $\|g\|$ is its cyclically reduced word length: the length of the shortest word in the generators of $\AG$ or their inverses that represents an element conjugate to $g$. 

\begin{proposition}
\label{keyA}
Let $\AG$ be a RAAG with defining graph $\G$ that is homogeneous of dimension 2. Then  for any minimal action $\alpha$ of $\AG$ on a 2-dimensional CAT(0)
rectangle complex, for any $g \in \AG$,
$l^\alpha_1(g) \leq 7M_1^\alpha \|g\|$.
\end{proposition}

\begin{proof}
Let $g \in \AG$. Let $g^\prime$ be a cyclically reduced element in the conjugacy class of $g$, so $\|g\|=\|g^\prime\|=|g^\prime|$. Letting $n=\|g\|$ and writing $g^\prime=g_1g_2\dots g_n$, where the $g_i$ are generators or inverses of generators, we have
\begin{align*}
l^\alpha_1(g)& =l^\alpha_1(g^\prime) \leq d_1(x_0, g^\prime.x_0)\\
& \leq d_1(x_0, g_1.x_0) + d_1 (g_1.x_0, g_1g_2.x_0) + \dots + d_1(g_1g_2\dots g_{n-1}.x_0, g_1g_2\dots g_{n-1}g_n.x_0)\\
&= \sum_{i=1}^{n}d_1(x_0, g_i.x_0)\\
& \leq \sum_{i=1}^{n}l^\alpha_1(g_i) + 2\sum_{i=1}^{n}d_1(x_0, \m(g_i))\\
& \leq nM_1^\alpha +2n\cdot3M_1^\alpha \\
& = 7nM_1^\alpha \\
& =7M_1^\alpha \|g\|
\end{align*}
\end{proof}

Proposition~\ref{keyA} concerns $l_1$ translation lengths; we now translate this proposition into an analogous proposition concerning CAT(0) translation lengths. Similar to $M_1^\alpha$, we define
$$  M^\alpha =\max_{d\in D}l^\alpha(d),$$
where $D$  is the set of all elements of $\AG $ of reduced word length less than or equal to 4.

\begin{proposition}
\label{key}
Let $\AG$ be a RAAG with defining graph $\G$ that is homogeneous of dimension 2. Then there exists a constant $C$ such that for any minimal action $\alpha$ of $\AG$ on a 2-dimensional CAT(0)
rectangle complex, for any $g \in \AG$,
$l^\alpha(g) \leq CM^\alpha \|g\|$.
\end{proposition}

\begin{proof}
This follows (with $C=7\sqrt2$) from Proposition~\ref{keyA} and the fact that for any element $h\in \AG$, $l^\alpha(h)\leq l_1^\alpha(h)\leq \sqrt2 l^\alpha$. The first inequality here is clear from the definitions. For the second inequality, recall that if $x$ is any point of $\m(h)$, then  $l_1^\alpha(h)=d_1(x,h.x)$ while $l^\alpha(h)=d(x, h.x)$; and it is clear from the fact that the rectangle complex is 2-dimensional that $d_1(x,h.x)\leq \sqrt 2d(x, h.x).$ 

\end{proof}

Let $\Omega$ be the set of conjugacy classes of $\AG$. Any translation length function $l^\alpha$ is constant on conjugacy classes, so we can view $l^\alpha$ as a function $l^\alpha:\Omega\rightarrow \mathbb{R}$, so $l^\alpha \in \mathbb{R}^\Omega$. 

Let $\Gamma$ be homogeneous of dimension 2. Let $\AG$ act properly, semi-simply by isometries on a 2-dimensional CAT(0) rectangle complex $X$. By definition, the action is minimal if $X$ is the union of the minsets of the $\Z^2$-subgroups of $\AG$. So if an action by $\AG$ is minimal, by definition it is proper, and every $e\neq g \in \AG$ acts as a semisimple isometry. Note that if $\alpha$ is a proper, semisimple action by isometries, then $l^\alpha \neq 0$, because if $l^\alpha= 0$ then every $g\in \AG$ fixes a point (since every $g$ is semi-simple), but then every $g=e$, since point stabilizers are finite and $\AG$ is torsion-free. So the set of translation length functions of minimal actions of a given RAAG on 2-dimensional CAT(0) rectangle complexes lies in $\R^\infty \setminus 0$.

\begin{theorem}
\label{main}
Let $\AG$ be a RAAG with defining graph $\G$ that is homogeneous of dimension 2. The space of projectivized translation length functions of minimal actions of $\AG$ on 2-dimensional CAT(0) rectangle complexes lies within a compact subspace of $P^\Omega$.
\end{theorem}

\begin{proof} 
Let $\alpha$ be a minimal action of $\AG$ on a 2-dimensional CAT(0) rectangle complex, and let $l^\alpha \in \mathbb{R}^\Omega$ be its translation length function. By Proposition~\ref{key}, for any $g \in \AG$, $l^\alpha(g) \leq CM^\alpha \|g\|$. Define ${l^\alpha}^\prime$ to be a scalar multiple of $l^\alpha$, ${l^\alpha}^\prime(g)=\dfrac{l^\alpha(g)}{CM^\alpha}$ (note that $M^\alpha \neq 0$ since $l^\alpha \neq 0$ and $l^\alpha(g) \leq CM^\alpha\|g\|$). Then (1) for any $g \in \AG$, ${l^\alpha}^\prime(g) \leq \|g\|$, and (2) there exists $d^* \in D$ such that ${l^\alpha}^\prime(d^*)\geq\dfrac{1}{C}$. (1) is clear, and (2) is true because by definition, there is a $d^*\in D$ such that $M^\alpha=l^\alpha(d^*)$, so that $d^*$ satisfies ${l^\alpha}^\prime(d^*)=\dfrac{{l^\alpha}(d^*)}{CM^\alpha}=\dfrac{1}{C}$.

Let $\Omega^D$ be the set of conjugacy classes $[d]$ that have a representative $d$ in $D$. Let $K$ be the subset of $ \mathbb{R}^\Omega \setminus 0$ defined by $K= \bigcup_{[d]\in \Omega^D} \left(\left[\frac{1}{C},\|d\| \right]_{[d]} \times \prod_{[g]\neq [d]}\left[0,\|g\| \right]_{[g]} \right)$. $K$ is compact since $\Omega^D$ is finite. Let $K^\prime$ be the image of $K$ in $\mathbb{P}^\Omega$. $K^\prime$ is compact because it is the image of a compact set under the projection map $\mathbb{R}^\Omega \setminus 0 \rightarrow \mathbb{P}^\Omega$. (1) and (2) show that ${l^\alpha}^\prime \in K$, so the image of $l^\alpha$ under projectivization (which is the same as the image of ${l^\alpha}^\prime$ under projectivization) lies in $K^\prime$. $K^\prime$ is independent of $\alpha$, so this shows that the set of projectivized translation length functions arising from minimal actions of $\AG$ on CAT(0) cubical complexes lies in a compact subspace of $\mathbb{P}^\Omega$.
\end{proof}

\section{Future Directions}

Recall that Theorem~\ref{main} is a statement about the set of minimal actions by a RAAG $\AG$ whose defining graph $\G$ is homogeneous of dimension 2. This condition on $\G$ ensures that $\AG$ acts minimally on 2-dimensional CAT(0) rectangle complexes. The natural analogue of this condition in higher dimensions is the requirement that $\G$ be homogeneous of dimension $n$; in this case $\AG$  acts minimally on $n$-dimensional CAT(0) rectangle complexes.

Recall that the proof of Theorem~\ref{main} depended on a reduction to arguments about $d_1$ distances. The ideas developed in the 2-dimensional case should prove useful in higher dimensions, where arguments concerning edge paths should be much easier than arguments concerning geodesics. It should be possible to use my techniques to prove analogues of both of Culler and Morgan's theorems for any RAAG $\AG$ whose defining graph is homogeneous of dimension $n$. That is, it should be possible to prove that the map from the set of minimal actions of $\AG$ on $n$-dimensional CAT(0) rectangle complexes to $\bbP^\infty$ is injective, and that the image of this map has compact closure in $\bbP^\infty$.

An analogue of Outer Space for a RAAG $\AG$ should have points that are minimal actions on CAT(0) rectangle complexes. Theorem~\ref{main} says that the closure of such a space in $\bbP^\infty$ is compact. It will be interesting to investigate the nature of the boundary points. We conjecture that, as in the free group case, these boundary points correspond to some degeneration of the actions. Can we describe the boundary points as actions on metric spaces?
If so, what kinds of metric spaces are involved? In what ways are they generalizations of CAT(0) rectangle complexes? In what ways can distances shrink or lengthen as the boundary is approached? 

In the case of a free group, actions on simplicial trees degenerate in the boundary into actions on $\R$-trees. Cohen and Lustig \cite{CL} defined the geometric notion of a very small action of a free group on an $\R$-tree, and they and Bestvina and Feighn \cite{CL,BF} showed that the compactification of Outer Space consists of all very small actions. In the case of a 2-dimensional RAAG, actions on CAT(0) rectangle complexes might degenerate in the boundary into actions on measured wall spaces, the natural higher-dimensional analogues of $\R$-trees.

\end{document}